\documentclass[11pt]{amsart}
\usepackage{amsfonts}
\usepackage{stmaryrd}  
\usepackage{graphicx}
\usepackage{epsfig}
\usepackage{amsmath}
\usepackage{amsthm}
\usepackage{amssymb}
\usepackage{float} 
\usepackage{multirow}
\usepackage{verbatim}
\usepackage{fancyhdr}
\usepackage{subfigure}
\usepackage{color}
\usepackage{bm}
\usepackage{enumerate}
\usepackage[numbers,sort&compress]{natbib}
\usepackage{diagbox}

\newcommand{\jl}{j-\frac 12}
\newcommand{\jr}{j+\frac 12}

\newcommand{\dd}{\mathrm{d}}

\newtheorem{lem}{Lemma}[section]
\newtheorem{exmp}{Example}
\newtheorem{thm}{Theorem}[section]
\newtheorem{prop}{Proposition}[section]
\numberwithin{equation}{section}
\newtheorem{rem}{Remark}[section]

\allowdisplaybreaks

\usepackage{soul}
\usepackage{xcolor}

\soulregister\cite7

\begin{document}

\title[Penalty DGM for ND problems]{Numerical analysis of a class of penalty 
discontinuous Galerkin methods for nonlocal diffusion problems}  

\author{Qiang Du}
\address{Department of Applied Physics and Applied Mathematics and Data Science Institute, Columbia University, New York, NY 10027, USA.}
\thanks{Q. Du's research is partially supported by  US National Science Foundation grant DMS-2012562.}
\email{qd2125@columbia.edu}

\author{Lili Ju}
\address{Department of Mathematics, University of South Carolina, Columbia, SC 29208, USA.}
\thanks{L. Ju's research is partially supported by US National Science Foundation grant DMS-2109633.}
\email{ju@math.sc.edu}

\author{Jianfang Lu}
\address{School of Mathematics, South China  University of Technology, Guangzhou,
Canton 510641, China.}
\thanks{J. Lu's research is partially supported by NSFC 11901213 and Science and Technology Program of Guangzhou 
2023A04J1300. }
\email{jflu@scut.edu.cn}

\author{Xiaochuan Tian}
\address{Department of Mathematics, University of California San Diego, La Jolla, CA 92093, USA.}
\thanks{X. Tian's research is partially supported by US National Science Foundation grant DMS-2111608.}
\email{xctian@ucsd.edu} 


\begin{abstract}
In this paper, we consider a class of discontinuous Galerkin (DG) methods 
for one-dimensional nonlocal diffusion (ND) problems.  
The nonlocal models, which are integral equations, are widely used in describing 
many physical phenomena with long-range interactions. 
The ND problem is the nonlocal analog of the classic diffusion problem, 
and as the interaction radius (horizon) vanishes, then the nonlocality disappears and the ND problem 
converges to the classic diffusion problem. Under certain conditions, the exact solution to the ND problem 
may exhibit discontinuities, setting it apart from the classic diffusion problem. 
Since the DG method shows its great advantages in resolving problems with discontinuities 
in computational fluid dynamics over the past several decades, it is natural to adopt the DG method to compute the ND problems. 
Based on \cite{DuCAMC2020}, we develop the DG methods with different penalty terms, 
ensuring that the proposed DG methods have local counterparts as the horizon vanishes. 
This indicates the proposed methods will converge to the existing DG schemes 
as the horizon vanishes, which is crucial for achieving {\it asymptotic compatibility}. 
Rigorous proofs are provided to demonstrate the stability, error estimates, and 
asymptotic compatibility of the proposed DG schemes. 
To observe the effect of the nonlocal diffusion, we also consider the time-dependent convection-diffusion problems
with nonlocal diffusion. We conduct several numerical experiments, including accuracy tests 
and Burgers' equation with nonlocal diffusion, and various horizons are taken to show the good performance 
of the proposed algorithm and validate the theoretical findings. 
\end{abstract}


\maketitle

\section{Introduction}
\label{sec_intro}

Nonlocal modeling has become quite popular in describing some physical phenomena 
involving nonlocal interactions of finite range in recent years. 
Unlike the classic local partial differential equation models, the nonlocal models 
can describe the physical phenomena in a setting with reduced regularity requirements 
and allow the singularities and discontinuities to occur naturally. Additionally, due to the finite 
interaction range, the nonlocal models would be more computationally efficient 
compared to some integro-differential equations characterized by an infinite range of interactions. 
The benefits of nonlocal models have led to their widespread use in various fields, such as  
crack and fracture in solid mechanics \cite{SillingJMPS2000, SillingAAM2010, DuM2AN2011}, 
traffic flows \cite{HuangSIAP2022}, nonlocal wave equation \cite{GuanNMPDE2015, DuCICP2018}, 
nonlocal convection-diffusion problems \cite{DuDCDSB2014, TianCMAME2015}, 
phase transitions \cite{DuSINUM2016, DuSINUM2019}, and image processing 
\cite{GilboaMMS2007, GilboaMMS2008}. 
Nonlocal volume-constrained diffusion models have a strong connection 
with the fractional Laplacian and fractional derivatives \cite{DuSIREV2012, DEliaAN2020}. 
The boundary conditions of the nonlocal diffusion (ND) problems 
are defined on a nonzero volume region outside the domain, serving as a natural extension of those in differential equation problems. 
Since most singular phenomena consist of both smooth and nonsmooth regions, 
a natural approach is to apply nonlocal models in the nonsmooth regions and local models in the smooth regions. 
Consequently, if local models are used near the boundary, the boundary conditions would remain the classic ones. 
This leads to the seamless coupling of the local and nonlocal models, 
see \cite{DuSINUM2018, DEliaJPNM2022} and the references cited therein. 

Currently, there are several research topics on nonlocal models, including 
 nonlocal vector calculus \cite{GunzburgerMMS2010, DuM3AS2013, DEliaJPNM2020}, 
nonlocal trace spaces \cite{TianSIAP2017, DuJFA2022}, 
nonlocal modeling and its mathematical investigations 
and numerical simulations, to name a few, \cite{DuM2AN2011, TianSINUM2013, DengM2AN2013, DuDCDSB2014, XuSINUM2014, TianCMAME2015, QiuJCP2015, DuSINUM2016, DuSINUM2017, DuCICP2018, DuSINUM2019, DuMC2019, DuFCM2020, HuangSIAP2022}. 
For a comprehensive literature review, readers can refer to \cite{DuSIREV2012, DuBook2019, DEliaAN2020} and the references therein. 
Among the various numerical methods for nonlocal problems, 
the discontinuous Galerkin (DG) finite element method is a natural 
choice due to its effectiveness in handling singularities and discontinuities. 
In 1973, Reed and Hill introduced the first DG method to solve the steady transport equation \cite{Reed1973}. 
Around 1990, Cockburn and Shu et al. combined the third-order total variation diminishing 
Runge-Kutta method \cite{ShuJCP1988, ShuSISSC1988} in temporal discretization with DG method in spatial discretization 
and successfully solved hyperbolic conservation laws 
 \cite{CockburnM2AN1991, CockburnMC1989, CockburnJCP1989, CockburnMC1990, CockburnJCP1998}. 
Since then, the DG method has gained significant attention and become widely used in many fields, such as 
aeroacoustics, oceanography, meteorology, electromagnetism, granular flows, turbulent flows, 
viscoelastic flows, magneto-hydrodynamics, 
oil recovery simulation, semiconductor device simulation, transport of contaminants in porous media 
 and weather forecasting, etc. 
Meanwhile, the study on the DG method for diffusion problems was developed independently in the 1970s 
and there exist various DG methods for diffusion problems. 
To name a few, there are symmetric interior penalty Galerkin method \cite{DouglasLNP1976}, 
nonsymmetric interior penalty Galerkin method \cite{RiviereCG1999}, 
Baumann-Oden's method \cite{BaumannCMAME1999}, 
Babu\v{s}ka-Zl\'{a}mal's method \cite{BabuskaSINUM1973}, 
local DG method \cite{CockburnSINUM1998}, ultra-weak DG method \cite{ChengMC2008}, 
recovery DG method \cite{vanLeerAIAA2005}, direct DG method \cite{LiuSINUM2009} 
and hybridizable DG method \cite{CockburnSINUM2009a}, sparse grid DG method \cite{WangJCP2016}, 
weak Galerkin method \cite{WangJCAM2013},  embedded DG method \cite{CockburnSINUM2009b}, etc. 
In particular, in \cite{ArnoldSINUM2002} Arnold et al. provided a general framework to analyze the DG methods 
with interior penalty and revealed the key aspects of constructing these methods. 

Even though many DG methods have been constructed and proposed for classic diffusion problems previously, 
 these DG methods cannot be trivially extended to ND problems. 
The major difficulty lies in the absence of differential operators 
in the nonlocal diffusion problems, 
which means integration by parts cannot be used, and jumps do not appear in the weak formulations. 
In \cite{ChenCMAME2011}, Chen and Gunzburger proposed a discontinuous finite element method 
for peridynamic models 
in the continuous finite element framework using the piecewise polynomial finite element space. 
They observed that the convergence of the numerical solutions depends on the choice of the horizon 
and mesh size when using the piecewise constant discretizations. 
In \cite{TianSINUM2015}, Tian and Du constructed a nonconforming DG method for nonlocal variational problems. 
However, none of the previously mentioned DG methods are {\it asymptotically compatible} (see, e.g., \cite{TianSINUM2014, TianSIREV2020}). 
In \cite{DuMC2019}, the authors developed a DG method for ND problems
 , which is a nonlocal analog of the local DG method \cite{CockburnSINUM1998} 
and achieved asymptotic compatibility. Later in \cite{DuCAMC2020}, 
the authors proposed a penalty DG method, which is a nonlocal analog of Babu\v{s}ka-Zl\'{a}mal's DG method \cite{BabuskaSINUM1973}. 
In this paper, we extend this work by constructing a more general DG method for ND problems. 
We provide a general framework for the penalty DG methods applied to ND problems 
and artificially construct the penalty terms involving jumps. With the artificial penalty terms, we are able to 
recover several aforementioned well-known DG methods. Theoretical results on boundedness, 
stability, and a priori error estimates are also provided. To observe the nonlocal diffusion effect, we consider 
a convection-diffusion problem with nonlocal diffusion. With the standard DG discretization for 
the convective term and the proposed methods for the ND term, we obtain the $L^2$-stability in the 
semi-discrete case. Several numerical examples, including accuracy tests, nonsmooth ND problems, 
and Burgers' equations with nonlocal diffusion, are presented to validate the good performance 
of the proposed algorithm.

\medskip
The rest of the paper is organized as follows. In Section \ref{sec_model_dgm}, we first introduce 
the nonlocal diffusion problem and its variational form, then propose the discontinuous Galerkin 
methods with some artificial penalty terms. In Section \ref{sec_theory}, we study the boundedness, 
stability, and a priori error estimates for the proposed schemes. The convection-diffusion model with 
nonlocal diffusion is considered and a semi-discrete analysis is given. 
In Section \ref{sec_numerics}, we show some numerical examples, including the smooth and nonsmooth problems, 
and the time-dependent convection-diffusion problem with nonlocal diffusion. 
Concluding remarks are given in Section \ref{sec_summary}.

\section{Method formulation}
\label{sec_model_dgm}

In this section, we first introduce the one-dimensional ND problem and its variational form. 
Next, we construct the penalty DG methods for the ND problem with artificial penalty terms. 
The penalties on the jumps not only ensure that the integral is well-defined but can also be adjusted 
to create nonlocal analogs of several existing DG methods for classic diffusion problems.

\subsection{Problem description}
\label{sec_problem_description}

 Consider the one-dimensional steady-state ND problem with the nonlocal volume constraint 
in the following 
\begin{eqnarray}
\label{eqn_model_tind}
\left\{
\begin{aligned}
&\mathcal{L}_\delta u = f_\delta, \ \  x \in \Omega \triangleq (a, b),  \\
&u = 0, \ \ \ \ x \in \Omega_\delta \triangleq [a-\delta,a] \cup [b,b+\delta], 
\end{aligned}
\right.
\end{eqnarray}
where $\delta >0$ is a constant. 
The operator $\mathcal{L}_\delta$ is defined as
$$
\mathcal{L}_\delta u (x) : = - 2 \int_{x-\delta}^{x + \delta} (u(y) - u(x)) 
\widehat{\gamma_\delta}(x,y) \, \dd y. 
$$
The kernel function $\widehat{\gamma_\delta}(x,y)$ is nonnegative and symmetric. 
For simplicity, in the paper, we consider a special case that 
\begin{eqnarray}
\label{eqn_bd_kernel}
\left\{ \begin{aligned}
& \widehat{\gamma_\delta}(x,y) = \gamma_\delta(s)  
= \gamma_\delta(-s), \; s = x - y,  \\  
& s^2 \gamma_\delta(s) \in L^1_{loc}(\mathbb{R}).  
\end{aligned}
\right.
\end{eqnarray}
The natural energy space associated with \eqref{eqn_model_tind} is 
$$
\mathcal{S} = \Big\{ v \in L^2\big( \widetilde{\Omega} \big) : 
\| v \|_{\mathcal{S}} < \infty, \ \ \text{$ v = 0$ 
on $\Omega_\delta$} \Big\}, 
$$
where $\widetilde{\Omega} = \Omega \cup \Omega_\delta$ and the semi-norm 
$\Vert v \Vert_{\mathcal{S}}$ is defined as
$$
\Vert v \Vert_{\mathcal{S}}^2 = 2 \int_0^\delta \gamma_\delta(s)  
\int_{\widetilde{\Omega}} \big( E_s^+ v(x) \big)^2 \, \dd x \dd s ,  
$$
where $E_s^+ w(x) = w(x + s) - w(x)$. 
In fact, the semi-norm $\Vert \cdot \Vert_{\mathcal{S}}$ is a norm on $\mathcal{S}$
(see e.g. \cite{TianSINUM2015}). 
Then the variation form of \eqref{eqn_model_tind} is as follows: 
\begin{align}
\label{variation_form}
\text{Find } u \in \mathcal{S} \text{ such that  } \quad B(u, v) = (f, v),\quad \forall 
\, v \in \mathcal{S},  
\end{align}
where the bilinear form is given as 
\begin{align} \label{eqn_bilinear_form}
B(u, v) = 2 \int_{0}^{\delta} \gamma_\delta(s) 
\int_{a - \delta}^{b + \delta}E_s^+ u(x) \, E_s^+ v(x) \, \dd x \dd s,   
\end{align}
and $(\cdot, \cdot)$ is the usual $L^2$ product on $\widetilde{\Omega}$. 
Note that the integral in $B(u, v)$ requires the values of $u$ outside $\widetilde{\Omega}$, 
therefore we take the zero extension of $u$ such that $u = 0 $ on $\widetilde{\Omega}^c$.


Since $s^2 \gamma_\delta(s) \in L^1_{loc}(\mathbb{R})$, without loss of generality 
we assume that 
$$
 \int_{-\delta}^{\delta} s^2 \gamma_\delta(s) \, \dd s = 1. 
$$
When $\delta \rightarrow 0$, the nonlocal diffusion problem 
\eqref{eqn_model_tind} becomes the heat equation with Dirichlet boundary condition 
as follows:
\begin{eqnarray}
\label{eqn_poisson}
\left\{\begin{aligned}
&- u_{xx} = f, \ \ x \in \Omega,\\
& u(a) = u(b) = 0.
\end{aligned}
\right.
\end{eqnarray}
We refer the readers to \cite{DuSIREV2012, DuBook2019, DEliaAN2020} for more details.


\subsection{Penalty discontinuous Galerkin methods}
\label{sec_dgm}

 To construct the penalty DG method, we first take the partition of the domain $\widetilde{\Omega}$ as $\mathcal{T}_h 
= \big\{ I_j = \big(x_{j-\frac 12}, x_{j+\frac 12} \big) \big\}_{j=-m+1}^{N+m}$, with 
\begin{align} \label{1dmesh}
x_{\frac 12} = a, \,\, x_{N+\frac 12} = b, \,\,
 x_{-m-\frac 12} \leq a - \delta < x_{-m + \frac 12}, \, \,
x_{N+m - \frac 12} < b + \delta \leq x_{N+m + \frac 12}. 
\end{align}
Assume the partition $\mathcal{T}_h$ is regular, i.e., there exists a constant 
$\nu > 0$ such that 
\begin{align} \label{mesh_regularity}
\nu h \leq \rho. 
\end{align}
where $h, \rho$ are given as 
\begin{align} \label{mesh_info}
h = \max_j h_j, \, \rho = \min_j h_j, \, h_j = x_{j+\frac 12} - x_{j-\frac 12}.
\end{align}
Then we define the finite element space  as
\begin{align}
\label{fes}
 V_h = V_h^k = \left\{v \in L^2\big( \widetilde{\Omega} \big): \ \ v|_{I_j}\in \mathcal{P}_k(I_j),\; j = 1, \cdots, N,  \,
 v|_{ \Omega_\delta} = 0 \right\},
\end{align}
 where $\mathcal{P}_k(I_j)$ is the space of polynomials on $I_j$ whose degrees are at most $k$.  
Following \cite{DuCAMC2020}, we divide $B(u, v)$ in \eqref{eqn_bilinear_form} into three parts as follows:  
\begin{align}
\label{eqn_bilinear_form1} 
B(u, v)  = B_1(u, v) + B_2(u, v) + B_3(u, v), \quad  \forall \, u, v \in \mathcal{S},  
\end{align}  
where the above three terms are given as 
\begin{align}
\label{eqn_bilinear_form2}
\begin{aligned}
B_1(u, v) & \, =2 \int_0^{\hat{h}} \gamma_\delta(s) \sum\limits_{j} 
\int_{I_{j, 1}^s} E_s^+ u(x) 
E_s^+ v(x) \,\dd x \dd s, \\
B_2(u, v) & \, = 2\int_0^{\hat{h}} \gamma_\delta(s) \sum\limits_{j}  
\int_{I_{j, 2}^s } E_s^+ u(x) 
E_s^+ v(x) \,\dd x \dd s,   \\
B_3(u, v) & \, = 2 \int_{\hat{h}}^\delta \gamma_\delta(s) \sum\limits_{j} \int_{I_j} 
E_s^+ u(x) E_s^+ v(x) \,\dd x \dd s ,   
\end{aligned}
\end{align} 
where $I_{j, 1}^s$ and $I_{j, 2}^s$ are given as 
$$
I_{j, 1}^s = \big( x_{j - \frac 12}, x_{j+\frac 12} - s \big), \quad
I_{j, 2}^s = \big( x_{j + \frac 12} - s, x_{j+\frac 12} \big), \quad
\hat{h} = \min \{\rho, \delta\}. 
$$
In \cite{DuCAMC2020}, we showed $B(u, v)$ may not be well-defined in the discrete space $V_h$ 
since $B_2(u, v)$ would cause troubles. Since 
$ I_{j, 2}^s = \big( x_{j + \frac 12} - s, x_{j+\frac 12} \big), $
then $u(x)$ and $u(x+s)$ are in the different elements, and this may lead to 
$ \gamma_\delta(s) (u(x+s) - u(x)) (v(x+s) - v(x))$ not integrable on $ (0, \hat{h})$ when 
$u, v \in V_h$. 
We now take the similar treatment in \cite{DuCAMC2020} and put the boundary terms together, 
then we can obtain the bilinear form at the discrete level in the following. 
\begin{align}
\label{eqn_bilinear_form3}
B_h(u_h, v_h)  =  E(u_h, v_h) + J(u_h, v_h) + \mu P(u_h, v_h), 
\end{align}
where $\mu > 0$ is taken to be large enough to ensure the stability 
and $E(u_h, v_h), \, P(u_h, v_h)$ are defined as 
\begin{align*}
\begin{aligned}
E(u_h, v_h) & \, = 2 \sum\limits_{j} \int_0^{\hat{h}} \gamma_\delta(s) 
\int_{I_{j, 1}^s} E_s^+ u_h(x) 
E_s^+ v_h(x) \, \dd x \dd s \\ 
& \quad + 2 \sum\limits_{j} \int_0^{\hat{h}} \gamma_\delta(s) 
\int_{I_{j, 2}^s} \Big(E_s^+ u_h(x) - [\![u_h]\!]_{j+\frac 12} \Big) 
\Big(E_s^+ v_h(x) - [\![v_h]\!]_{j+\frac 12} \Big) \, \dd x \dd s \\
& \quad + 2 \sum\limits_{j} \int_{\hat{h}}^\delta \gamma_\delta(s) \int_{I_j} E_s^+ u_h(x) 
E_s^+ v_h(x) \,  \dd x \dd s , \\
P(u_h, v_h) & \,  = \int_0^{\hat{h}} s^2 \gamma_\delta(s) \, \dd s \sum\limits_{j} 
 [\![u_h]\!]_{j+\frac 12} [\![v_h]\!]_{j+\frac 12}, 
\end{aligned}
\end{align*}
where $[\![w]\!]_{j+\frac 12} = w\big(x_{j+\frac 12}^+ \big) - w\big(x_{j+\frac 12}^- \big)$. 
If we define $g_v(x, s)$ in the following: 
\begin{align} \label{eqn_gv}
g_v(x, s) =  \left\{
\begin{aligned}
& E_s^+ v(x) - [\![v]\!]_{j+\frac 12},  \quad  \text{when } 
x \in I_{j, 2}^s, \; s \in (0, \hat{h}),  \\
& E_s^+ v(x),  \hspace{2cm}  \text{elsewhere.}
\end{aligned}
\right.
\end{align}
Then we have 
\begin{align*}
E(u_h, v_h) & \, =  2 \int_0^{\delta} \gamma_\delta(s) 
\sum\limits_{j} \int_{I_j } g_{u_h}(x, s) g_{v_h}(x, s) \, \dd x \dd s.  
\end{align*}
$J(u_h, v_h)$ consists of boundary terms obtained by an integration by parts 
in the nonlocal sense and different choices of $J(u_h, v_h)$ and 
$\mu$ lead to various DG schemes. 
For instance, we show three kinds of DG formulations in the following 
by choosing different $J(u_h, v_h)$ and $\mu$, and the local limits 
of these DG methods have been known for many years. 

\medskip
\noindent$\bullet$ Nonlocal version of Babu\v{s}ka-Zl\'amal  method \cite{BabuskaSINUM1973} (nBZ) : 
\begin{align}
\label{eqn_nbz}
\begin{aligned}
J(u_h, v_h) & \, = 0, \\
\mu & \, = O(h^{-2k - 1}). 
\end{aligned}
\end{align}
Here $k$ is the degree of the polynomials in $V_h^k$. \\
$\bullet$ Nonlocal version of IP method  \cite{DouglasLNP1976} (nIP) : 
\begin{align}
\label{eqn_nip}
\begin{aligned}
J(u_h, v_h) & \, = 2 \sum\limits_{j} [\![v_h]\!]_{j+\frac 12} \int_0^{\hat{h}} 
\gamma_\delta(s) \int_{I_{j, 2}^s } g_{u_h}(x, s)
 \, \dd x \dd s \\
& \quad + 2 \sum\limits_{j} [\![u_h]\!]_{j+\frac 12} \, \int_0^{\hat{h}} 
\gamma_\delta(s) \int_{I_{j, 2}^s }  g_{v_h}(x, s) \, \dd x \dd s,   \\
\mu = & \, O(h^{-1}). 
\end{aligned}
\end{align}
$\bullet$ Nonlocal version of NIPG method  \cite{RiviereCG1999} (nNIPG) :  
\begin{align}
\label{eqn_nnipg}
\begin{aligned}
J(u_h, v_h) =& \, 2 \sum\limits_{j} [\![v_h]\!]_{j+\frac 12} \int_0^{\hat{h}} 
\gamma_\delta(s) \int_{I_{j, 2}^s } g_{u_h}(x, s)
 \, \dd x \dd s \\
& - 2 \sum\limits_{j} [\![u_h]\!]_{j+\frac 12} \, \int_0^{\hat{h}} 
\gamma_\delta(s) \int_{I_{j, 2}^s }  g_{v_h}(x, s) \, \dd x \dd s,  \\
\mu = & \, O(h^{-1}). 
\end{aligned}
\end{align}
In particular, the nBZ method has already been analyzed in \cite{DuCAMC2020}. In the rest of 
the paper, we focus on the two DG schemes, nIP and nNIPG. 

 \begin{rem}
 \label{rmk1}
As previously mentioned, several DG methods exist for diffusion problems in the literature. 
Although only three DG methods are listed above, this approach can be extended 
to several penalty DG methods in {\rm \cite{ArnoldSINUM2002}} without substantial difficulty. 
It is worth noting that the penalty term $P(u_h, v_h)$ may require suitable modifications of the lifting operator. 
We refer readers to {\rm \cite{ArnoldSINUM2002}} for more details. 
 \end{rem}

Now we present the penalty DG formulation for the ND problem \eqref{eqn_model_tind} as follows:
\begin{align}
\label{dgscm}
\text{Find $u_h \in V_h$ such that } \quad B_h(u_h, v_h) = (f_\delta, v_h), \quad 
\forall \, v_h \in V_h, 
\end{align}
 where $B_h(u_h, v_h)$ is given in \eqref{eqn_bilinear_form3}, 
 and $J(u_h, v_h)$ and $\mu$ can be taken either one from \eqref{eqn_nbz},  \eqref{eqn_nip} 
 or \eqref{eqn_nnipg}.

\section{Boundedness, Stability and a priori error estimates}
\label{sec_theory}

In this section, we consider the boundedness and stability of the DG methods. 
The nIP (\eqref{dgscm} and \eqref{eqn_nip}) and nNIPG schemes (\eqref{dgscm} and \eqref{eqn_nnipg}) are consistent, while 
the nBZ scheme (\eqref{dgscm} and \eqref{eqn_nbz})  is not. 
To control the inconsistency error in the nBZ method, the so-called superpenalty technique was applied 
to estimate the inconsistent term, see, e.g., \cite{ArnoldSINUM2002, DuCAMC2020}. 
In this section, we focus on nIP and nNIPG methods, deriving the 
boundedness, stability, and a priori error estimates for these two DG methods. 
 Throughout this section, we let $C>0$ represent a generic 
constant independent of $h$ and $\delta$ but with possibly different values. 
Now let us define the semi-norms for $v \in V(h) \triangleq V_h + \mathcal{S}$ as follows:
\begin{align}
\label{eqn_seminorm}
\begin{aligned}
|v|_{E, h}^2 =&\,   2 \sum\limits_{j} \int_0^{\delta} \gamma_\delta(s) 
\int_{I_j} g_{v}(x, s)^2 \, \dd x \dd s , \\
|v|_{J, h}^2 = & \, 2 \sum\limits_{j} h_j \int_0^{\hat{h}} \gamma_\delta(s) 
\, \dfrac{1}{s} \int_{I_{j, 2}^s } g_{v}(x, s)^2 \, \dd x \dd s, \\
|v|_{P, h}^2 =& \,  \int_0^{\hat{h}} s^2 \gamma_\delta(s) \,\dd s \sum\limits_{j} 
 [\![v]\!]_{j+\frac 12}^2 . 
\end{aligned}
\end{align}

To consider the boundedness and stability for the bilinear form $B_h(\cdot, \cdot)$, 
we define the norm for $v \in V(h) $ as follows:
\begin{align}
\label{eqn_norm}
||| v |||^2 = |v|_{E, h}^2 + |v|_{J,h}^2 + \mu\, |v|_{P, h}^2. 
\end{align}
To see that $||| \cdot |||$ is a norm on $V_h$, we have the following 
proposition in \cite{DuCAMC2020}. 
 \begin{prop} \label{prop_discrete_poincare} 
 For the general kernels $\gamma_\delta$ satisfying \eqref{eqn_bd_kernel}, it holds that 
 for some constant $C>0$ independent of $\delta$ and $h$ such that 
\begin{align}
\label{eqn_poincare}
 \| v_h \|_{L^2} \leq C \, ||| v_h ||| , \quad \forall \,v_h \in V_h. 
\end{align}
\end{prop}

\subsection{Boundedness}
\label{sec_boundedness}

The boundedness is straightforward after we define the norm $||| \cdot ||| $ 
on $V(h)$. For instance, we consider the term $J(u_h, v_h)$ in \eqref{eqn_nip}. 
In fact, $\forall \, v, w \in V_h + \mathcal{S}$, by Cauchy-Schwarz inequality we have 
\begin{align}
\label{eqn_bound_1}
\begin{aligned}
&\sum\limits_{j} [\![v]\!]_{j+\frac 12} \int_0^{\hat{h}} \gamma_\delta(s) 
\int_{I_{j, 2}^s} g_w(x, s) \, \dd x \dd s \\
 & \quad \leq
 \Big(  \int_0^{\hat{h}} s^2 \gamma_\delta(s) \,\dd s  
\sum\limits_{j} \frac{1}{h_j} [\![v]\!]_{j+\frac 12}^2 \Big)^{\frac 12} \, 
\Big( \sum\limits_{j} h_j \int_0^{\hat{h}} \gamma_\delta(s) \frac{1}{s}
 \int_{I_{j, 2}^s} g_w(x, s)^2 \, \dd x \dd s \Big)^{\frac 12} \\
 & \quad \leq \, (2 \rho )^{- \frac 12} \, |v|_{P, h} \, |w|_{J,h},   
\end{aligned}
\end{align} 
where $\rho$ is defined in \eqref{mesh_info}. 
Similarly, we have 
\begin{align}
\label{eqn_bound_2}
\sum\limits_{j} [\![w]\!]_{j+\frac 12} \int_0^{\hat{h}} \gamma_\delta(s)
\int_{I_{j, 2}^s} g_v(x, s) \, \dd x \dd s \leq  ( 2 \rho )^{- \frac 12} \, |w|_{P, h} \, |v|_{J,h}.
\end{align}
From \eqref{eqn_bound_1}, \eqref{eqn_bound_2} and \eqref{eqn_nip}, we obtain
\begin{align}
\label{eqn_bound_3}
J(v, w) \leq \, \sqrt{2}  \rho^{- \frac 12} ( |v|_{P, h} \, |w|_{J,h} +  |w|_{P, h} \, |v|_{J,h} ). 
\end{align}
Therefore, for any $\mu \geq 1/\rho$, by Cauchy-Schwarz inequality we have
\begin{align}
\label{eqn_boundedness}
\begin{aligned}
B_h(v, w) & \, \leq \, |v|_{E, h} |w|_{E, h} 
+  \sqrt{2} \rho^{- \frac 12} (|w|_{P, h} \, |v|_{J,h} +  |v|_{P, h} \, |w|_{J,h} ) 
+ \mu \, |v|_{P, h} \, |w|_{P, h}  \\
& \, \leq 2 \, ||| v ||| \, |||w |||, \quad \forall \, v, w \in V(h). 
\end{aligned}
\end{align} 

\subsection{Stability}
\label{sec_stability}

 We now show the DG methods \eqref{dgscm} are stable, i.e. $\exists \, C_s > 0$ 
independent of $h$ and $\delta$ such that 
\begin{align} \label{eqn_stability}
 B_h(v_h, v_h) \geq C_s ||| v_h |||^2, \quad \forall \, v_h \in V_h. 
\end{align}
From the definition of the bilinear form $B_h(\cdot, \cdot)$ in \eqref{eqn_bilinear_form3}, 
for any $v_h \in V_h$ we have 
\begin{align*}
B_h(v_h, v_h)  = | v_h |_{E, h}^2 + J(v_h, v_h) + \mu\, | v_h |_{P, h}^2. 
\end{align*} 
therefore, we can obtain the stability once $J(v_h, v_h)$ is controlled.  
A Lemma is presented below, which is crucial in deriving the stability of the nIP method. 
\begin{lem} \label{lem_sta_cond}
For any $ v_h \in V_h$, there exists a constant $C_0 > 0$  independent of $h$ and $\delta$, 
such that 
\begin{align} \label{eqn_sta_cond} 
|v_h|_{J,h}^2 \leq C_0 \,  |v_h|_{E, h}^2 .   
\end{align}
\end{lem}
The proof of Lemma \ref{lem_sta_cond} is given in Appendix \ref{sec_proof_lem_sta_cond}. 

If \eqref{eqn_sta_cond} holds, 
upon using $J(v_h, v_h) \leq 2 \, \rho^{- \frac 12} \, 
|v_h|_{P, h} \, |v_h|_{J,h} $ from \eqref{eqn_bound_3} and Cauchy-Schwarz inequality 
we then have 
\begin{align} \label{eqn_sta_1}
\begin{aligned}
B_h(v_h, v_h) & \,  \geq  | v_h |_{E, h}^2 - 2 \, \rho^{- \frac 12} \, 
|v_h|_{P, h} \, |v_h|_{J,h}  + \mu\, | v_h |_{P, h}^2 \\ 
& \, \geq | v_h |_{E, h}^2  - 2 C_0 \rho^{-1} |v_h|_{P, h}^2 - (2 C_0)^{-1} |v_h|_{J,h}^2
+ \mu\, | v_h |_{P, h}^2 \\  
& \, \geq \frac{1}{2}| v_h |_{E, h}^2 + ( \mu - 2 C_0 \rho^{-1} ) |v_h|_{P, h}^2  \\
& \, \geq (2C_0 + 2)^{-1}  ( | v_h |_{E, h}^2 + | v_h |_{J, h}^2 + \mu | v_h |_{P, h}^2 ) \\
& \quad +  ( \mu - 2 C_0 \rho^{-1} - \mu (2C_0 + 2)^{-1}  ) \, |v_h|_{P, h}^2. 
\end{aligned}
\end{align}
Therefore, if we take $\mu$ sufficiently large, for instance 
$ \mu \geq (1 - (2 C_0 + 2)^{-1})^{-1 } 2 C_0 \rho^{-1} $, such that  
$$  
 \mu - 2 C_0 \rho^{-1} - \mu (2C_0 + 2)^{-1}  \geq 0, 
$$ 
then from \eqref{eqn_sta_1} we immediately  obtain desired result \eqref{eqn_stability} 
with $C_s = (2C_0 + 2)^{-1}$.

\subsection{An a priori error estimate}
\label{sec_error_estimate}

 In the previous section, we have obtained the boundedness and stability of the DG methods 
 \eqref{dgscm}, we now consider the consistency and approximation error. 
 In the error estimates, we first make the assumption 
that the exact solution $u$ is smooth, which means $u$ does not have discontinuities inside the domain $\widetilde{\Omega}$ so that the continuous interpolation can be defined.  
We take the continuous interpolant 
$u_I \in V_h$ of the exact solution $u$ so that $u - u_I$ will be zero at the element 
interfaces, then we have the following approximation properties \cite{DuCAMC2020}: 
\begin{align} 
\label{eqn_approx} 
\| u - u_I \|_{L^2} \leq C  h^{k+1}\, |u|_{H^{k+1}}  \quad {\rm and} \quad  
|||u - u_I |||  \leq C  h^{ k}  \, |u|_{H^{k+1}},  
\end{align}
where $C > 0$ is independent of $h$ and $\delta$, and $\displaystyle |u|_{H^{k+1}} = \Big( \int_{\widetilde{\Omega}} (\partial_x^{k+1} u )^2 \dd x \Big)^{1/2}$ is the semi-norm. We refer the readers to \cite{Susanne2008} for more details.

In \cite{DuCAMC2020}, we know that with enough smoothness of the exact solution $u$, 
the bilinear form $B(u, v_h)$ is well-defined and $B(u, v_h) = (f_\delta, v_h)$. 
For the nIP method and nNIPG method, if they are consistent, i.e.  $B_h(u, v_h) = B(u, v_h)$, 
we then have 
$$
B_h(u, v_h) = B(u, v_h) = (f_\delta, v_h) = B_h(u_h, v_h), \quad \forall \, v_h \in V_h.   
$$
In fact, we can check the consistency with some calculations that 
\begin{align*}
E(u, v_h) + J(u, v_h) \, &= 2 \sum\limits_{j} \int_0^{\hat{h}} \gamma_\delta(s) 
\int_{I_{j, 1}^s} E_s^+ u(x) 
E_s^+ v_h(x) \, \dd x \dd s \\ 
& \quad + 2 \sum\limits_{j} \int_0^{\hat{h}} \gamma_\delta(s) 
\int_{I_{j, 2}^s} E_s^+ u(x) 
\Big(E_s^+ v_h(x) - [\![v_h]\!]_{j+\frac 12} \Big) \, \dd x \dd s \\
& \quad + 2 \sum\limits_{j} \int_{\hat{h}}^\delta \gamma_\delta(s) \int_{I_j} E_s^+ u(x) 
E_s^+ v_h(x) \,  \dd x \dd s \\
& \quad + 2 \sum\limits_{j} [\![v_h]\!]_{j+\frac 12} \int_0^{\hat{h}} 
\gamma_\delta(s) \int_{I_{j, 2}^s } E_s^+ u(x) \, \dd x \dd s \\
 & = B_1(u, v_h) + B_2(u, v_h) + B_3(u, v_h) = B(u, v_h). 
 \end{align*}
Note that we use the smoothness of $u$ that $[\![u]\!]_{j+\frac 12} = 0, \, \forall \, j$, which also leads to  
$P(u, v_h) = 0$, thus we have $B_h(u, v_h) = E(u, v_h) + J(u, v_h) + \mu P(u, v_h) = B(u, v_h)$. 
Therefore, the nIP method is consistent, as well as the conventional IP method, 
and it also holds for the nNIPG method. 

So far, we have obtained the boundedness, stability, 
 approximation properties and consistency of the DG methods \eqref{dgscm}, we now 
proceed to derive the error estimates. 
 For the consistent DG method nIP and nNIPG, we have 
\begin{align}
\label{eqn_errest_1} 
\begin{aligned}
C_s ||| u_I - u_h |||^2 & \,\leq  B_h(u_I - u_h, u_I - u_h) \\
& \, = B_h(u_I - u, u_I - u_h) + B_h(u - u_h, u_I - u_h) \\
& \, \leq 2 \, ||| u_I - u ||| \, ||| u_I - u_h |||.   
\end{aligned}
\end{align}
With \eqref{eqn_approx}, \eqref{eqn_errest_1} and triangle inequality, we then obtain
\begin{align} \label{eqn_errest_2}
\begin{aligned}
||| u - u_h ||| & \, \leq ||| u - u_I ||| + ||| u_I - u_h ||| \\ 
& \, \leq \Big( 1+ \frac{2}{C_s} \Big) ||| u - u_I ||| \leq C h^k \, |u|_{H^{k+1}}. 
\end{aligned}
\end{align}
For the nBZ method, we also have a similar result as \eqref{eqn_errest_2}. 
Since the nBZ method is not consistent, we need to make an extra effort to 
control the inconsistent term. In fact, the nBZ method relies on a rather heavy 
penalty and it may require some preconditioning since the condition number of the stiff matrix would be 
large due to this penalty. For more details, one can refer to \cite{DuCAMC2020}. 
We now summarize the above results as follows. 

\begin{thm} \label{thm_errest}
Consider the DG methods \eqref{dgscm} for solving the ND problem 
\eqref{eqn_model_tind}, with the finite element space $V_h$ defined in \eqref{fes} 
and the degrees of the piecewise polynomials $k \geq 1$. With suitable 
penalties on the jumps of element interfaces, such as in \eqref{eqn_nip}, \eqref{eqn_nnipg}
and \eqref{eqn_nbz}, there exists a unique numerical solution $u_h \in V_h$ to 
\eqref{dgscm}. Assume the exact solution of \eqref{eqn_model_tind} 
$u \in H^{k+1}\big( \widetilde{\Omega} \big)$, then we have the following error estimate:  
$$||| u - u_h ||| \leq C h^k \| u \|_{H^{k+1}}. $$
\end{thm}

\begin{rem} \label{rmk_errest}
From \eqref{eqn_errest_2}, we can see the $L^2$ error $\| u - u_h \|_{L^2}$ 
is controlled by the interpolation 
error $||| u - u_I |||$ for the nIP and nNIPG methods. For integrable kernels with a fixed horizon 
$\delta$, the discrete energy norm $||| \cdot |||$ is equivalent to the $L^2$ norm 
{\rm (}see e.g. {\rm \cite{MengeshaPRSESA2014}}{\rm )}. Thus, in this situation, we have the error estimate:  
$$||| u - u_h ||| \leq C(\delta) h^{k + 1} \| u \|_{H^{k+1}}. $$
\end{rem}

\subsection{Asymptotic compatibility}
\label{sec_ac}

In the continuous level, when the horizon $\delta \to 0$, the solution $u$ of the ND problem 
\eqref{eqn_model_tind} converges to the solution $u_{loc}$ of the corresponding local problem 
\eqref{eqn_poisson} (see e.g. \cite{DuSIREV2012}), i.e. 
\begin{align} \label{eqn_u_to_uloc}
\| u - u_{loc} \|_{L^2} \to 0, \quad {\rm as } \,\, \delta \to 0. 
\end{align}  
It is desirable to preserve such a limiting behavior in the numerical approximations, 
termed as {\it asymptotic compatibility} \cite{TianSINUM2014}, such that $u_h \to u_{loc}$ when 
 $\delta, h \to 0$ simultaneously. 
In \cite{TianSINUM2013}, Tian and Du studied several existing numerical schemes 
and showed some of the numerical discretizations might not preserve such a limiting behavior. 
Later in \cite{TianSINUM2014}, Tian and Du established an abstract mathematical 
framework for the numerical studies of a class of parametrized problems. 
We show that the DG method \eqref{dgscm} is also asymptotically compatible 
under some appropriate conditions, stated in the following theorem. 

\begin{thm} \label{thm_ac}
Consider the DG methods \eqref{dgscm} for solving the ND problem 
\eqref{eqn_model_tind}, with the finite element space $V_h^k$ defined in \eqref{fes}, $k \geq 1$. 
Assume the exact solution $u \in H^{1 + \beta}\big( \widetilde{\Omega} \big)$, $0 < \beta < 1$, 
 and $\| u \|_{H^{1 + \beta}\big( \widetilde{\Omega}\big) } $ is uniformly bounded with respect to the parameter $\delta$. 
Then the DG methods \eqref{dgscm} are asymptotically compatible, i.e. 
\begin{align} \label{eqn_ac}  
\| u_h - u_{loc} \|_{L^2} \to 0, \quad \text{as } \delta, h \to 0. 
\end{align}
\end{thm}

\begin{rem}
The assumption of the solution $u \in H^{1 + \beta}\big( \widetilde{\Omega} \big)$, $0 < \beta < 1$ is valid for the ND problem \eqref{eqn_model_tind} with a 
truncated fractional kernel. The $H^{1+\beta}$ regularity of  solution of such problem with any fixed $\delta$ was shown in  {\rm \cite[Theorem 3.4]{BurkovskaJMAA2019}}.  However, it is quite difficult to obtain a $\delta$-independent bound for the $H^{1 + \beta}$ norm of $u$ and thus the uniform estimate in $\delta$ of $H^{1+\beta}$ regularity still remains an open question to be explored. 
\end{rem}

To prove the Theorem \ref{thm_ac}, we first introduce a result in \cite[Chap. 14]{Susanne2008} as follows, 
which plays a key role in obtaining the asymptotic compatibility in the fractional space.  

\begin{lem} \label{lem_interpolation}
Suppose $T$ is a linear operator that maps $H^p(\Omega)$ to $L^2(\Omega)$ and 
$H^q(\Omega)$ to $L^2(\Omega)$, where $p, q$ are some positive integers.
then $T$ maps $H^{(1-\theta)p + \theta q}(\Omega)$ to $L^2(\Omega)$, $\forall \, 0 < \theta < 1$. Moreover,
\begin{align*}
\| T\|_{ H^{(1-\theta)p + \theta q} \to L^2 }  \leq \|T \|^{1 - \theta}_{H^p \to L^2} \, \|T \|^{\theta}_{H^q \to L^2} . 
\end{align*}
\end{lem}
 In the proof of Theorem \ref{thm_ac}, we would like to consider the interpolation 
between $\mathcal{L}(H^1_0(\Omega) \cap H^2(\Omega), L^2(\Omega))$ 
and $\mathcal{L}(H_0^1(\Omega), L^2(\Omega))$, and this is the case when $p = 1, q = 2$ in Lemma \ref{lem_interpolation}. 

 \begin{proof}[Proof of Theorem \ref{thm_ac}:]
 	From the discrete Poincar\'e's inequality \eqref{eqn_poincare} and \eqref{eqn_errest_1}, we have 
$$
\| u_I - u_h \|_{L^2} \leq C \,||| u_I - u_h ||| \leq  C \,||| u_I - u |||. 
$$
Therefore, by the triangle inequality and \eqref{eqn_poincare}, we can obtain  
$$
\| u - u_h \|_{L^2} \leq \| u - u_I \|_{L^2} + \| u_I - u_h \|_{L^2} 
\leq C \,||| u_I - u ||| \, ,
$$
where $u_I \in V_h$ is the continuous interpolant of the exact solution $u$. 
For $u \in H^2 \big( \widetilde{\Omega} \big)$, from \eqref{eqn_errest_2} we have 
\[
\| u - u_h \|_{L^2} \leq  C h \| u \|_{H^2}. 
\]
Now we claim that 
\begin{align} \label{err_bd_H1}
\| u - u_h \|_{L^2} \leq C \,\| u\|_{H^1}. 
\end{align}
If \eqref{err_bd_H1} holds true, then we can denote a linear operator $Tu: = u - u_h$ and obtain 
\begin{align*}
& \| T \|_{H^1 \to L^2} \leq C, \quad  \| T \|_{H^2 \to L^2} \leq C h. 
\end{align*}
With $p = 1, q = 2$ in the Lemma \ref{lem_interpolation}, we conclude
\begin{align*}
& \| T \|_{H^{1 + \beta} \to L^2} \leq C h^{\beta} ,  \quad \forall \, 0 < \beta < 1. 
\end{align*}
which indicates $\|u - u_h\|_{L^2} \leq C h^{\beta} \| u \|_{H^{1 + \beta}}$, 
where $C$ is independent of $\delta$ and $h$. Therefore, $\|u - u_h\|_{L^2} \to 0$ as $h \to 0$. 
since $\| u_{loc} - u \|_{L^2} \to 0$ as $\delta \to 0$, then 
\[
\| u_{loc} - u_h \|_{L^2} \leq \| u_{loc} - u \|_{L^2} + \| u - u_h \|_{L^2} \to 0, \quad \text{as } \delta, h \to 0 .
\]
This indicates the DG methods \eqref{dgscm} 
are asymptotically compatible for the exact solution $u \in H^{1 + \beta}\big( \widetilde{\Omega} \big)$, 
$ 0 < \beta < 1$. 

It remains to prove \eqref{err_bd_H1}.  
For the consistent DG method nIP and nNIPG, for any $v_h \in V_h$ we have 
\begin{align} \label{err_bd_H1_1} 
\begin{aligned}
C_s ||| v_h - u_h |||^2 & \,\leq  B_h( v_h - u_h, v_h - u_h) \\
& \, = B_h(v_h - u, v_h - u_h) + B_h( u - u_h, v_h - u_h) \\
& \, \leq 2 \, ||| v_h - u ||| \, ||| v_h - u_h |||, 
\end{aligned}
\end{align}
which implies 
$$
||| v_h - u_h ||| \leq \frac{2}{C_s} ||| v_h - u |||, \quad \forall \, v_h \in V_h. 
$$
Together with the Proposition \ref{prop_discrete_poincare}, we have 
\[
\| v_h - u_h \|_{L^2} \leq C \,||| v_h - u_h ||| \leq C \,||| v_h - u |||. 
\]
By triangle inequality and above inequality, we obtain 
\begin{align} \label{err_bd_H1_2} 
\begin{aligned}
\| u - u_h \|_{L^2} & \, \leq \| u - v_h \|_{L^2} + \| v_h - u_h \|_{L^2} \\
& \, \leq \| u - v_h \|_{L^2} + C \,||| v_h - u |||. 
\end{aligned}
\end{align}
Since \eqref{err_bd_H1_2} holds true for any $v_h \in V_h$, then we take $v_h = 0$ 
and obtain 
\[
\| u - u_h \|_{L^2}  \leq  \| u \|_{L^2} + C \,||| u ||| =  \| u \|_{L^2} + C \,\| u \|_{\mathcal{S}} \leq C \, \| u\|_{H^1}, 
\]
which proves the claim \eqref{err_bd_H1}.  The estimate $\|u \|_{\mathcal{S}} \leq C \|u\|_{H^1}$ 
used in the last inequality can be found in, e.g., \cite{BBM2001}. 

\end{proof}

\subsection{Application to convection-diffusion problems}
\label{sec_convection_diffusion}

In this subsection, we consider the time-dependent convection-diffusion problem 
with nonlocal diffusion as follows:
\begin{align} \label{eqn_cd}
\left\{ 
\begin{aligned} 
& u_t + f(u)_x + \sigma \mathcal{L}_\delta u = f_s, \quad x \in \Omega, \; t > 0, \\
& u(x, 0) = u_0(x),  \quad  x \in \Omega,  
\end{aligned} \right. 
\end{align}
with periodic or compactly supported boundary conditions. $f(u)$ is the flux function 
and $f_s = f_s(x, t)$ is the source function. 
In the numerical approximation of convection-dominated problems, 
one of the computational challenges is the sharp transitions of the numerical solution. 
In particular, when $\sigma = 0$ the solution of \eqref{eqn_cd} may evolve into 
shock discontinuities even with the smooth initial condition. There are many studies 
on the DG discretization of $f(u)_x$, see e.g. \cite{Shu2009} and the references therein. 
Now assume we have the partition of the domain the same as in \eqref{1dmesh}, we then construct 
the semi-discrete DG methods for \eqref{eqn_cd}: seek $u_h(\cdot, t) \in V_h$ such that 
\begin{align} \label{dgscm_td}
\int_{I_j} (u_h)_t v_h \,\dd x + A_j(u_h, v_h) + \sigma B_{h, j}(u_h, v_h) = \int_{I_j} f_s v_h \,\dd x, 
\quad \forall \, v_h \in V_h, 
\end{align}
where $A_j(u_h, v_h)$ is defined as 
\begin{align*}
A_j(u_h, v_h) = \hat{f}_{\jr} (v_h)_{\jr}^- - \hat{f}_{\jl} (v_h)_{\jl}^+ - \int_{I_j} f(u_h) (v_h)_x \,\dd x
\end{align*}
with $\hat{f}_{\jr}$ is the monotone flux and $ \sum_j B_{h, j}(u_h, v_h) = B_h(u_h, v_h) $
with $B_h(u_h, v_h)$ defined in \eqref{eqn_bilinear_form3}. 
Take $v_h = u_h$ in \eqref{dgscm_td} and sum it over $j$, we have 
\[
\frac{\dd}{2 \dd t} \| u_h(\cdot, t) \|_{L^2}^2 + \sum_j A_j(u_h, u_h) + \sigma B_h(u_h, u_h) = (f_s, u_h). 
\] 
Denote $\displaystyle F(u) = \int^u f(s) \, \dd s$, we then have 
\begin{align*}
 \sum_j A_j(u_h, u_h) & \,=  \sum_j  \Big( \hat{f}_{\jr} (u_h)_{\jr}^- - \hat{f}_{\jl} (u_h)_{\jl}^+ 
 - \big( F((u_h)_{\jr}^-) - F((u_h)_{\jl}^+ ) \big) \Big) \\
 & \, = \sum_j  \Big( \hat{f}_{\jr} (u_h)_{\jr}^- - \hat{f}_{\jr} (u_h)_{\jr}^+ 
 - \big( F((u_h)_{\jr}^-) - F((u_h)_{\jr}^+ ) \big) \Big)  \\
 & \, = \sum_j \int_{(u_h)_{\jr}^-}^{(u_h)_{\jr}^+} \Big( f(u) - \hat{f}\big( (u_h)_{\jr}^-, (u_h)_{\jr}^+ 
\big) \Big) \, . 
\end{align*}
By the above semi-discrete analysis, we can obtain the $L^2$-boundedness of the numerical solution, 
stated in the following theorem. 
\begin{thm} \label{L2_bound}
For the DG methods \eqref{dgscm_td}, the numerical solution satisfies 
\begin{align} \label{eqn_L2_bound} 
\frac{\dd}{2 \dd t} \| u_h(\cdot, t) \|_{L^2}^2 + \sum_j \Theta_j + \sigma B_h(u_h, u_h) = (f_s, u_h), 
\end{align}
where $\Theta_j$ is given as 
\begin{align*}
\Theta_j = \int_{(u_h)_{\jr}^-}^{(u_h)_{\jr}^+} \Big( f(u) - \hat{f}\big( (u_h)_{\jr}^-, (u_h)_{\jr}^+ 
\big) \Big) \, \dd u \geq 0.
\end{align*} 
\end{thm} 
Thanks to the monotone flux, we are able to obtain $\Theta_j \geq 0$. We refer the readers to e.g. 
\cite{Shu2009} for more details. From the previous analysis, we have 
$B_h(u_h, u_h) \geq C_s ||| u_h |||^2 \geq C \| u_h \|^2$. Therefore, when $f_s = 0$, we can see that the numerical solution decays exponentially as time evolves.  

\begin{rem} \label{rmk_imex} 
For the time-dependent convection-diffusion problems, we recommend treating the diffusion term 
implicitly to have a relaxed CFL condition. 
Particularly, when both convection and diffusion co-exist, we could use the implicit-explicit {\rm (}IMEX{\rm )} 
Runge-Kutta method, such that the convection term could be treated explicitly 
and the diffusion term could be treated implicitly. 
\end{rem}

 \section{Numerical experiments}
\label{sec_numerics}

In this section, we show some numerical results to validate the theoretical 
results presented in the previous section. For simplicity, we take the uniform mesh, 
i.e. $h_j = h$. We consider the kernel function 
as 
$$
\gamma_{\delta}(s) = \frac{3 - \alpha}{2 \delta^{3 - \alpha}} |s|^{-\alpha} \quad \text{ on } 
(-\delta, \delta). 
$$
Then $s^2 \gamma_\delta(s)$ is integrable for $\alpha < 3$. In the numerical tests, 
we take $\alpha = 1/2, \, 5/2$ such that $\gamma_\delta$ could have different 
kinds of singularities. We use the five-point Gauss-Legendre quadrature when computing the integrals for $s>\hat{h}$. And for those integrals that $s$ close to $0$, 
we use the exact integration because they are the improper integrals involving singularities. 
For both nIP and nNIPG methods, we take the penalty parameter $\mu$ as $\mu = 5/h$ unless otherwise specified. 
We choose different values of $\delta$ in the numerical tests to show the good performance 
of the proposed numerical methods. For the time-dependent problems, we adopt the nIP method 
for nonlocal diffusion term, and the time discretization method is the 4th order 
implicit-explicit Runge-Kutta method with 6 stages \cite{CalvoANM2001}. Since we treat the convective term explicitly and the diffusive term implicitly, we take the CFL condition as $\tau = O(h)$ where $\tau$ is the time step.  Specifically, the CFL number is taken as $\frac{0.3}{ 2k + 1} $ for the numerical simulation of the time-dependent convection-diffusion equation.

\begin{exmp} \label{exp1}
For the steady-state problem \eqref{eqn_model_tind},
we take the source term  as
$$
f_\delta(x) = - 2 \int_{-\delta}^\delta  \gamma_\delta(s) \big( g(x + s) - g(x) \big) \,
\dd s, \quad x \in (0, \pi), 
$$
where $g(x)$ is defined by
\begin{align*}
g(x) = \left\{
\begin{aligned}
\sin^6(x), & \qquad x \in (0, \pi), \\
0, & \qquad  \text{elsewhere.} 
\end{aligned}
\right.
\end{align*}
Thus the exact solution is $u(x) = g(x)$. The computational domain is $\Omega = (0, \pi)$. 
\end{exmp}

From the results reported in Tables \ref{tab_exp1_nip_k1} - \ref{tab_exp1_nip_k3}, we can see the optimal order of convergence
for nIP scheme \eqref{dgscm} and \eqref{eqn_nip} with various $\alpha$ and $\delta$ as the mesh is refined. 
We also observe the designed order of convergence for nNIPG scheme \eqref{dgscm} and \eqref{eqn_nnipg} for the results reported 
in Tables \ref{tab_exp1_nnipg_k1} - \ref{tab_exp1_nnipg_k3}. Note that the order is not optimal for an even 
degree in the classic NIPG method, which coincides with the case $\delta = 10^{-6}$ in Table \ref{tab_exp1_nnipg_k2}. However, for $\gamma_\delta(s)$ is locally integrable 
($\alpha = 1/2$) with fixed horizon $\delta = \pi/6 $, the energy norm is equivalent 
to the $L^2$ norm. Then, from Remark 3.1, we have the optimal convergence order of 3 
in the energy norm. The optimal rate of convergence can also be found in the case $\gamma_\delta(s)$ is locally integrable with horizon $\delta = \sqrt{h}$.

\begin{table}[!htbp] \small
\caption{\label{tab_exp1_nip_k1} $L^2$ errors and convergence orders produced by the nIP scheme 
\eqref{dgscm} and \eqref{eqn_nip} when $k = 1$ in Example 
\ref{exp1}.}\centering 
\begin{tabular}{|c|c||cc|cc|cc|cc|}
  \hline
\multirow{2}{*}{$\alpha$}
& \multirow{2}{*}{$N$}
& \multicolumn{2}{|c|}{$\delta = 10^{-6} $}
  & \multicolumn{2}{|c|}{$\delta = \pi/6$} & \multicolumn{2}{|c|}{$\delta = 2.5 h$} 
  & \multicolumn{2}{|c|}{$\delta = \sqrt{h}$}
 \\
  \cline{3-10}
 &  & $L^2$ error & order & $L^2$ error & order & $L^2$ error & order
 & $L^2$ error & order  \\ \hline\hline
& 24   & 3.996E-03 &  --       &  1.697E-03  &  --         &   1.706E-03  &  --        &  1.703E-03  &  --         \\
& 36   & 1.803E-03 & 1.963 &  7.483E-04 &  2.019  &   7.516E-04  &  2.022 &  7.502E-04  & 2.021   \\
& 48   & 1.019E-03 & 1.982 &  4.199E-04 &  2.008  &   4.214E-04  &  2.011 &  4.204E-04  & 2.013   \\
$ \frac 12$
& 60   & 6.540E-04 & 1.989 &  2.685E-04 &  2.004  &   2.693E-04  &  2.007 &  2.688E-04  & 2.004   \\
& 72   & 4.548E-04 & 1.993 &  1.864E-04 &  2.002  &   1.868E-04  &  2.005 &  1.865E-04  & 2.005   \\
& 84   & 3.344E-04 & 1.995 &  1.369E-04 &  2.002  &   1.372E-04  &  2.003 &  1.370E-04  & 2.003   \\ 
& 96   & 2.562E-04 & 1.996 &  1.048E-04 &  2.001  &   1.050E-04  &  2.002 &  1.049E-04  & 2.001   \\ 
   \hline
& 24   & 3.996E-03 &  --       &  1.998E-03  &  --         &   2.129E-03  &  --        &  2.096E-03  &  --         \\
& 36   & 1.803E-03 & 1.963 &  8.435E-04 &  2.126  &   9.417E-04  &  2.012 &  8.980E-04  & 2.090   \\
& 48   & 1.019E-03 & 1.982 &  4.613E-04 &  2.098  &   5.288E-04  &  2.006 &  4.939E-04  & 2.078   \\
$\frac 52$
& 60   & 6.540E-04 & 1.989 &  2.900E-04 &  2.080  &   3.381E-04  &  2.004 &  3.112E-04  & 2.071   \\
& 72   & 4.548E-04 & 1.993 &  1.990E-04 &  2.067  &   2.347E-04  &  2.002 &  2.136E-04  & 2.064   \\
& 84   & 3.344E-04 & 1.995 &  1.449E-04 &  2.058  &   1.724E-04  &  2.002 &  1.554E-04  & 2.060   \\ 
& 96   & 2.562E-04 & 1.996 &  1.102E-04 &  2.051  &   1.320E-04  &  2.001 &  1.181E-04  & 2.056   \\ 
   \hline
\end{tabular}
\end{table}

\begin{table}[!htbp] \small
\caption{\label{tab_exp1_nip_k2} $L^2$ errors and convergence orders produced by the nIP scheme 
\eqref{dgscm} and \eqref{eqn_nip} when $k = 2$ in Example 
\ref{exp1}.}\centering 
\begin{tabular}{|c|c||cc|cc|cc|cc|}
  \hline
\multirow{2}{*}{$\alpha$}
& \multirow{2}{*}{$N$}
& \multicolumn{2}{|c|}{$\delta = 10^{-6} $}
  & \multicolumn{2}{|c|}{$\delta = \pi/6$} & \multicolumn{2}{|c|}{$\delta = 2.5 h$} 
  & \multicolumn{2}{|c|}{$\delta = \sqrt{h}$}
 \\
  \cline{3-10}
 &  & $L^2$ error & order & $L^2$ error & order & $L^2$ error & order
 & $L^2$ error & order  \\ \hline\hline
& 24   & 1.797E-04 &  --       &  1.012E-04  &  --         &   1.049E-04  &  --        &  1.041E-04  &  --         \\
& 36   & 5.358E-05 & 2.984 &  2.936E-05 &  3.052  &   3.166E-05  &  2.954 &  3.090E-05  & 2.996   \\
& 48   & 2.266E-05 & 2.992 &  1.211E-05 &  3.079  &   1.345E-05  &  2.977 &  1.294E-05  & 3.025   \\
$ \frac 12$
& 60   & 1.161E-05 & 2.995 &  6.078E-06 &  3.089  &   6.906E-06  &  2.986 &  6.586E-06  & 3.026   \\
& 72   & 6.724E-06 & 2.997 &  3.458E-06 &  3.094  &   4.003E-06  &  2.991 &  3.773E-06  & 3.056   \\
& 84   & 4.236E-06 & 2.998 &  2.146E-06 &  3.096  &   2.524E-06  &  2.993 &  2.370E-06  & 3.016   \\ 
& 96   & 2.838E-06 & 2.998 &  1.419E-06 &  3.096  &   1.692E-06  &  2.995 &  1.575E-06  & 3.061   \\ 
   \hline
& 24   & 9.746E-05 &  --       &  7.991E-05  &  --         &   7.997E-05  &  --        &  7.995E-04  &  --         \\
& 36   & 2.855E-05 & 3.028 &  2.366E-05 &  3.002  &   2.368E-05  &  3.002 &  2.367E-05  & 3.002   \\
& 48   & 1.200E-05 & 3.014 &  9.978E-05 &  3.001  &   9.985E-05  &  3.001 &  9.980E-05  & 3.001   \\
$ \frac 52$
& 60   & 6.131E-05 & 3.008 &  5.108E-06 &  3.001  &   5.112E-06  &  3.001 &  5.109E-06  & 3.001   \\
& 72   & 3.545E-06 & 3.006 &  2.956E-06 &  3.000  &   2.958E-06  &  3.000 &  2.956E-06  & 3.001   \\
& 84   & 2.231E-06 & 3.004 &  1.861E-06 &  3.000  &   1.863E-06  &  3.000 &  1.862E-06  & 3.000   \\ 
& 96   & 1.494E-06 & 3.003 &  1.247E-06 &  3.000  &   1.248E-06  &  3.000 &  1.247E-06  & 3.000   \\ 
   \hline
\end{tabular}
\end{table}

\begin{table}[!htbp] \small
\caption{\label{tab_exp1_nip_k3} $L^2$ errors and convergence orders produced by the nIP scheme 
\eqref{dgscm} and \eqref{eqn_nip} when $k = 3$ in Example 
\ref{exp1}.}\centering 
\begin{tabular}{|c|c||cc|cc|cc|cc|}
  \hline
\multirow{2}{*}{$\alpha$}
& \multirow{2}{*}{$N$}
& \multicolumn{2}{|c|}{$\delta = 10^{-6} $}
  & \multicolumn{2}{|c|}{$\delta = \pi/6$} & \multicolumn{2}{|c|}{$\delta = 2.5 h$} 
  & \multicolumn{2}{|c|}{$\delta = \sqrt{h}$}
 \\
  \cline{3-10}
 &  & $L^2$ error & order & $L^2$ error & order & $L^2$ error & order
 & $L^2$ error & order  \\ \hline\hline
& 24   & 1.189E-05 &  --       &  2.668E-06  &  --         &   2.672E-06  &  --        &  2.672E-06  &  --         \\
& 36   & 1.357E-06 & 5.353 &  5.198E-07 &  4.034  &   5.206E-07  &  4.034 &  5.204E-07  & 4.035   \\
& 48   & 3.576E-07 & 4.635 &  1.637E-07 &  4.016  &   1.639E-07  &  4.017 &  1.638E-07  & 4.017   \\
$ \frac 12$
& 60   & 1.339E-07 & 4.403 &  6.692E-08 &  4.009  &   6.699E-08  &  4.010 &  6.696E-08  & 4.010   \\
& 72   & 6.112E-08 & 4.300 &  3.224E-08 &  4.006  &   3.226E-08  &  4.007 &  3.225E-08  & 4.007   \\
& 84   & 3.194E-08 & 4.209 &  1.739E-08 &  4.004  &   1.740E-08  &  4.005 &  1.740E-08  & 4.004   \\ 
& 96   & 1.834E-08 & 4.157 &  1.019E-08 &  4.003  &   1.020E-08  &  4.004 &  1.019E-08  & 4.004   \\ 
   \hline
& 24   & 1.189E-05 &  --       &  3.182E-06  &  --         &   3.185E-06  &  --        &  3.184E-06  &  --         \\
& 36   & 1.357E-06 & 5.353 &  6.338E-07 &  3.980  &   6.345E-07  &  3.979 &  6.342E-07  & 3.980   \\
& 48   & 3.576E-07 & 4.635 &  2.011E-07 &  3.990  &   2.014E-07  &  3.990 &  2.012E-07  & 3.991   \\
$ \frac 52$
& 60   & 1.339E-07 & 4.403 &  8.246E-08 &  3.994  &   8.259E-08  &  3.994 &  8.251E-08  & 3.994   \\
& 72   & 6.112E-08 & 4.300 &  3.979E-08 &  3.996  &   3.986E-08  &  3.995 &  3.982E-08  & 3.996   \\
& 84   & 3.194E-08 & 4.209 &  2.149E-08 &  3.997  &   2.154E-08  &  3.995 &  2.150E-08  & 3.996   \\ 
& 96   & 1.834E-08 & 4.157 &  1.260E-08 &  3.997  &   1.264E-08  &  3.990 &  1.261E-08  & 3.995   \\ 
   \hline
\end{tabular}
\end{table}

\begin{table}[!htbp] \small
\caption{\label{tab_exp1_nnipg_k1} $L^2$ errors and convergence orders produced by the nNIPG scheme 
\eqref{dgscm} and \eqref{eqn_nnipg} when $k = 1$ in Example 
\ref{exp1}.}\centering 
\begin{tabular}{|c|c||cc|cc|cc|cc|}
  \hline
\multirow{2}{*}{$\alpha$}
& \multirow{2}{*}{$N$}
& \multicolumn{2}{|c|}{$\delta = 10^{-6} $}
  & \multicolumn{2}{|c|}{$\delta = \pi/6$} & \multicolumn{2}{|c|}{$\delta = 2.5 h$} 
  & \multicolumn{2}{|c|}{$\delta = \sqrt{h}$}
 \\
  \cline{3-10}
 &  & $L^2$ error & order & $L^2$ error & order & $L^2$ error & order
 & $L^2$ error & order  \\ \hline\hline
& 24   & 2.107E-03 &  --       &  1.704E-03  &  --         &   1.711E-03  &  --        &  1.709E-03  &  --         \\
& 36   & 9.325E-04 & 2.011 &  7.498E-04 &  2.025  &   7.519E-04  &  2.028 &  7.513E-04  & 2.026   \\
& 48   & 5.237E-04 & 2.005 &  4.204E-04 &  2.011  &   4.213E-04  &  2.014 &  4.208E-04  & 2.015   \\
$ \frac 12$
& 60   & 3.349E-04 & 2.003 &  2.687E-04 &  2.006  &   2.691E-04  &  2.008 &  2.690E-04  & 2.006   \\
& 72   & 2.325E-04 & 2.002 &  1.865E-04 &  2.004  &   1.867E-04  &  2.006 &  1.866E-04  & 2.006   \\
& 84   & 1.708E-04 & 2.002 &  1.369E-04 &  2.003  &   1.371E-04  &  2.004 &  1.370E-04  & 2.003   \\ 
& 96   & 1.307E-04 & 2.001 &  1.048E-04 &  2.002  &   1.049E-04  &  2.003 &  1.049E-04  & 2.002   \\ 
   \hline
& 24   & 2.107E-03 &  --       &  1.710E-03  &  --         &   1.716E-03  &  --        &  1.714E-03  &  --         \\
& 36   & 9.325E-04 & 2.011 &  7.527E-04 &  2.024  &   7.567E-04  &  2.020 &  7.548E-04  & 2.023   \\
& 48   & 5.237E-04 & 2.005 &  4.218E-04 &  2.013  &   4.244E-04  &  2.010 &  4.230E-04  & 2.013   \\
$ \frac 52$
& 60   & 3.349E-04 & 2.003 &  2.695E-04 &  2.008  &   2.713E-04  &  2.005 &  2.702E-04  & 2.008   \\
& 72   & 2.325E-04 & 2.002 &  1.869E-04 &  2.006  &   1.883E-04  &  2.004 &  1.874E-04  & 2.007   \\
& 84   & 1.708E-04 & 2.002 &  1.372E-04 &  2.005  &   1.382E-04  &  2.003 &  1.376E-04  & 2.005   \\ 
& 96   & 1.307E-04 & 2.001 &  1.050E-04 &  2.004  &   1.058E-04  &  2.002 &  1.053E-04  & 2.004   \\ 
   \hline
\end{tabular}
\end{table}

\begin{table}[!htbp] \small
\caption{\label{tab_exp1_nnipg_k2} $L^2$ errors and convergence orders produced by the nNIPG scheme 
\eqref{dgscm} and \eqref{eqn_nnipg} when $k = 2$ in Example 
\ref{exp1}.}\centering 
\begin{tabular}{|c|c||cc|cc|cc|cc|}
  \hline
\multirow{2}{*}{$\alpha$}
& \multirow{2}{*}{$N$}
& \multicolumn{2}{|c|}{$\delta = 10^{-6} $}
  & \multicolumn{2}{|c|}{$\delta = \pi/6$} & \multicolumn{2}{|c|}{$\delta = 2.5 h$} 
  & \multicolumn{2}{|c|}{$\delta = \sqrt{h}$}
 \\
  \cline{3-10}
 &  & $L^2$ error & order & $L^2$ error & order & $L^2$ error & order
 & $L^2$ error & order  \\ \hline\hline
& 24   & 5.449E-04 &  --       &  1.085E-04  &  --         &   1.136E-04  &  --        &  1.123E-04  &  --         \\
& 36   & 2.398E-04 & 2.024 &  3.168E-05 &  3.036  &   3.541E-05  &  2.876 &  3.348E-05  & 2.984   \\
& 48   & 1.345E-04 & 2.012 &  1.312E-05 &  3.065  &   1.565E-05  &  2.839 &  1.409E-05  & 3.008   \\
$ \frac 12$
& 60   & 8.591E-05 & 2.007 &  6.600E-06 &  3.077  &   8.417E-06  &  2.779 &  7.179E-06  & 3.023   \\
& 72   & 5.961E-05 & 2.005 &  3.761E-06 &  3.084  &   5.134E-06  &  2.711 &  4.127E-06  & 3.036   \\
& 84   & 4.377E-05 & 2.003 &  2.337E-06 &  3.088  &   3.416E-06  &  2.642 &  2.592E-06  & 3.018   \\ 
& 96   & 3.350E-05 & 2.003 &  1.546E-06 &  3.091  &   2.422E-06  &  2.577 &  1.726E-06  & 3.047   \\ 
   \hline
& 24   & 4.259E-04 &  --       &  3.580E-04  &  --         &   4.305E-04  &  --        &  4.128E-04  &  --         \\
& 36   & 1.870E-04 & 2.030 &  1.311E-04 &  2.478  &   1.899E-04  &  2.018 &  1.655E-04  & 2.254   \\
& 48   & 1.047E-04 & 2.015 &  6.397E-05 &  2.494  &   1.065E-04  &  2.009 &  8.640E-05  & 2.259   \\
$ \frac 52$
& 60   & 6.688E-05 & 2.009 &  3.662E-05 &  2.500  &   6.810E-05  &  2.006 &  5.218E-05  & 2.260   \\
& 72   & 4.639E-05 & 2.006 &  2.321E-05 &  2.502  &   4.726E-05  &  2.004 &  3.456E-05  & 2.259   \\
& 84   & 3.406E-05 & 2.004 &  1.578E-05 &  2.503  &   3.470E-05  &  2.003 &  2.440E-05  & 2.259   \\ 
& 96   & 2.607E-05 & 2.003 &  1.130E-05 &  2.503  &   2.656E-05  &  2.002 &  1.805E-05  & 2.258   \\ 
   \hline
\end{tabular}
\end{table}

\begin{table}[!htbp] \small
\caption{\label{tab_exp1_nnipg_k3} $L^2$ errors and convergence orders produced by the nNIPG scheme 
\eqref{dgscm} and \eqref{eqn_nnipg} when $k = 3$ in Example 
\ref{exp1}.}\centering 
\begin{tabular}{|c|c||cc|cc|cc|cc|}
  \hline
\multirow{2}{*}{$\alpha$}
& \multirow{2}{*}{$N$}
& \multicolumn{2}{|c|}{$\delta = 10^{-6} $}
  & \multicolumn{2}{|c|}{$\delta = \pi/6$} & \multicolumn{2}{|c|}{$\delta = 2.5 h$} 
  & \multicolumn{2}{|c|}{$\delta = \sqrt{h}$}
 \\
  \cline{3-10}
 &  & $L^2$ error & order & $L^2$ error & order & $L^2$ error & order
 & $L^2$ error & order  \\ \hline\hline
& 24   & 9.843E-06 &  --       &  2.678E-06  &  --         &   2.683E-06  &  --        &  2.682E-06  &  --         \\
& 36   & 1.915E-06 & 4.037 &  5.208E-07 &  4.038  &   5.217E-07  &  4.039 &  5.213E-07  & 4.039   \\
& 48   & 6.027E-07 & 4.019 &  1.639E-07 &  4.019  &   1.641E-07  &  4.020 &  1.640E-07  & 4.020   \\
$ \frac 12$
& 60   & 2.462E-07 & 4.011 &  6.697E-08 &  4.011  &   6.706E-08  &  4.012 &  6.701E-08  & 4.011   \\
& 72   & 1.186E-07 & 4.008 &  3.225E-08 &  4.007  &   3.229E-08  &  4.008 &  3.227E-08  & 4.008   \\
& 84   & 6.396E-08 & 4.005 &  1.740E-08 &  4.005  &   1.741E-08  &  4.006 &  1.740E-08  & 4.005   \\ 
& 96   & 3.747E-08 & 4.004 &  1.019E-08 &  4.004  &   1.020E-08  &  4.004 &  1.020E-08  & 4.004   \\ 
   \hline
& 24   & 9.843E-06 &  --       &  4.049E-06  &  --         &   4.311E-06  &  --        &  4.242E-06  &  --         \\
& 36   & 1.915E-06 & 4.037 &  7.450E-07 &  4.175  &   8.337E-07  &  4.052 &  7.933E-07  & 4.135   \\
& 48   & 6.027E-07 & 4.019 &  2.271E-07 &  4.129  &   2.617E-07  &  4.028 &  2.435E-07  & 4.106   \\
$ \frac 52$
& 60   & 2.463E-07 & 4.011 &  9.091E-08 &  4.103  &   1.067E-08  &  4.019 &  9.771E-08  & 4.092   \\
& 72   & 1.186E-07 & 4.008 &  4.315E-08 &  4.087  &   5.130E-08  &  4.018 &  4.640E-08  & 4.084   \\
& 84   & 6.396E-08 & 4.005 &  2.302E-08 &  4.076  &   2.760E-08  &  4.022 &  2.474E-08  & 4.080   \\ 
& 96   & 3.747E-08 & 4.004 &  1.337E-08 &  4.068  &   1.612E-08  &  4.029 &  1.435E-08  & 4.079   \\ 
   \hline
\end{tabular}
\end{table}

\begin{exmp} \label{exp2} 
Consider a non-smooth case for \eqref{eqn_model_tind}. We take the locally integrable kernel 
$\gamma_\delta$ that $\delta =1/8$, $\alpha = 1/2$. $g(x)$ is taken as 
\begin{align*}
g(x) = \left\{
\begin{aligned}
1, & \qquad x \in \Big( \frac 14, \frac 34 \Big), \\
0, & \qquad  \text{elsewhere.} 
\end{aligned}  \right.
\end{align*} 
With the definition of $f_\delta$ similar in the Example {\rm \ref{exp1}}, we have the solution $u(x) = g(x)$. 
The computational domain is $\Omega = (0, 1)$. 
\end{exmp} 

\begin{figure}[!htbp]
\subfigure[$k = 1$]{ 
\includegraphics[width=0.33\textwidth]{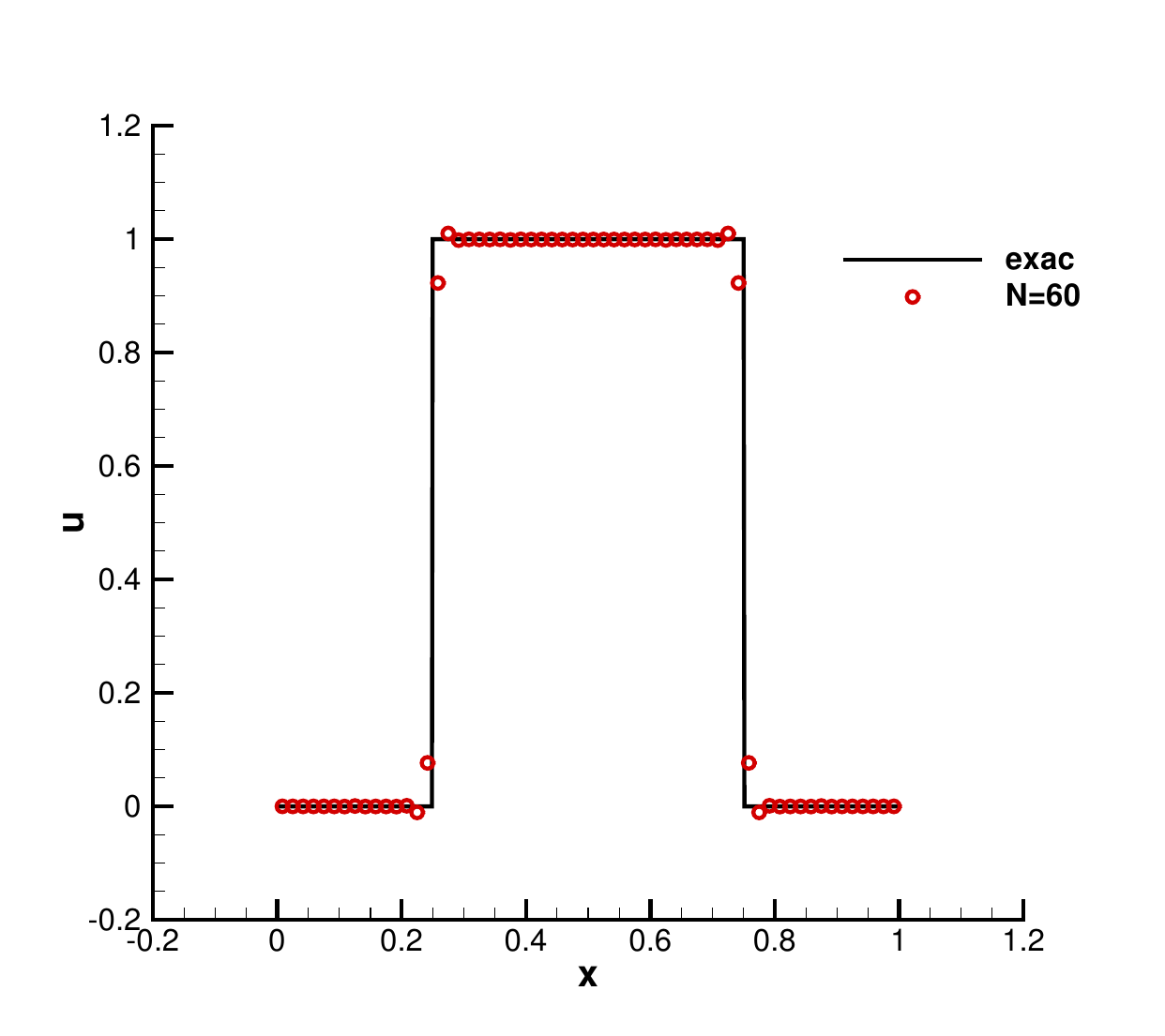} }\hspace{-0.5cm}
\subfigure[$k = 2$]{
\includegraphics[width=0.33\textwidth]{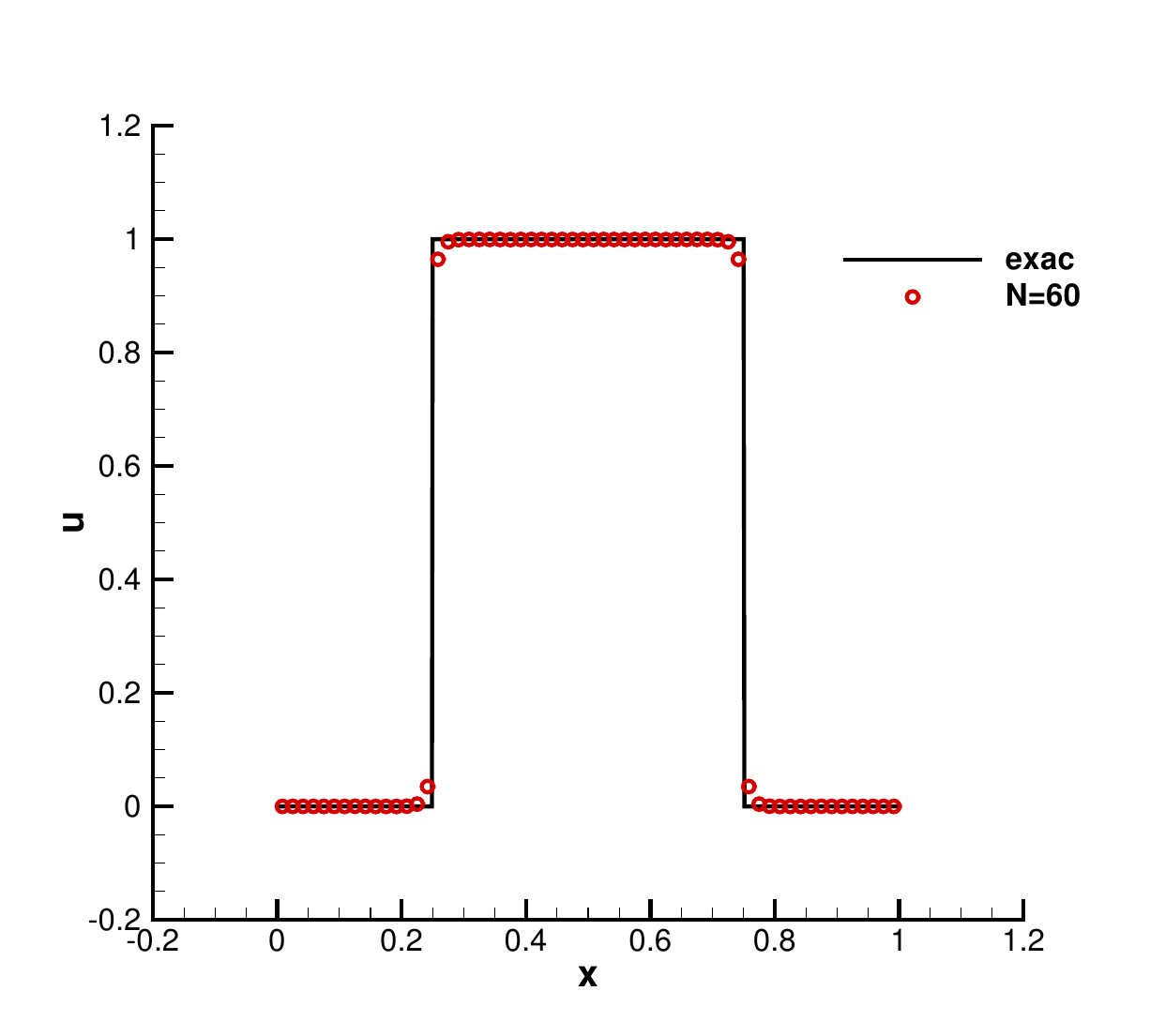} }\hspace{-0.5cm}
\subfigure[$k = 3$]{
\includegraphics[width=0.33\textwidth]{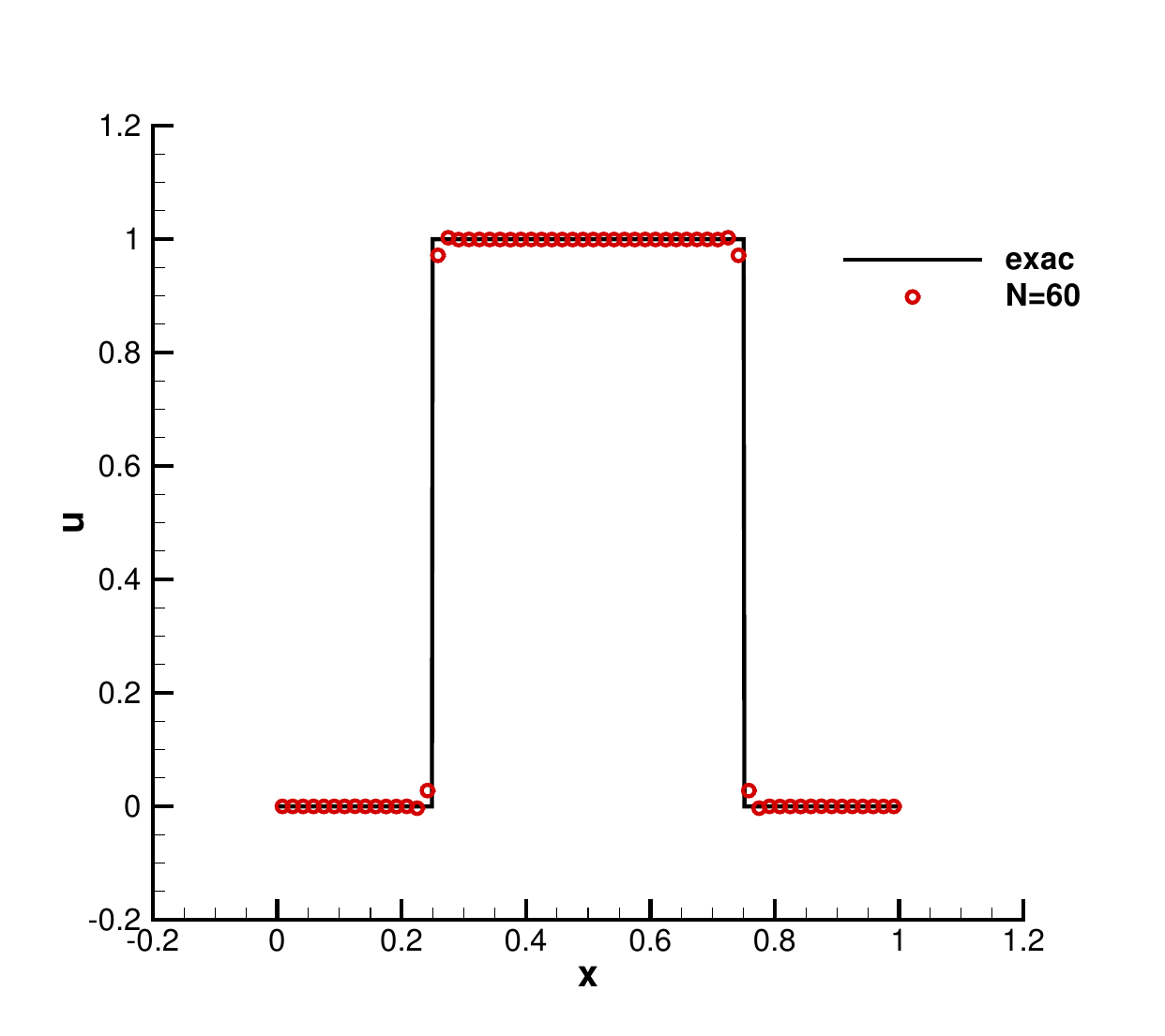} }
\caption{ \label{fig_exp2} 
Plots of the numerical solution in Example \ref{exp2} produced by the nIP 
scheme \eqref{dgscm} and \eqref{eqn_nip}, 
$\delta = 1/8, \alpha = 1/2$, $N = 60$. 
Solid line: exact solution. Red circles: numerical solutions.} 
\end{figure}

 Figure \ref{fig_exp2} plots of the numerical solution  produced by the nIP scheme \eqref{dgscm} and \eqref{eqn_nip}, 
 from which we can see there are no obvious oscillations in the numerical solutions, 
indicating the proposed algorithm can handle the singularities and discontinuities well.

\begin{exmp} \label{exp3} 
In this example, we consider the time-dependent convection-diffusion problems 
\eqref{eqn_cd} with $f(u) = u$ and the periodic boundary condition. 
The coefficient $\sigma $ is taken to be $1/2$. With a suitable choice of source 
function $f_s$, we have the exact solution as $u(x, t) = e^{-t} \sin^6(x)$. The computational 
domain $\Omega = (0, \pi)$ and the final time is $T = 2.2$. 
\end{exmp} 

From the results reported in  Tables \ref{tab_exp3_k1} - \ref{tab_exp3_k3}, we show the error and order of convergence 
of $\| u - u_h \|$ produced by the scheme 
\eqref{dgscm_td} and \eqref{eqn_nip}. For $k = 3$, we take the penalty parameter $\mu = 7/h$. 

\begin{table}[!htbp] \small
\caption{\label{tab_exp3_k1} $L^2$ errors and convergence orders produced by the scheme 
\eqref{dgscm_td} and \eqref{eqn_nip} when $k = 1$ in Example \ref{exp3}, 
$\sigma = 1/2$.}
\centering 
\begin{tabular}{|c|c||cc|cc|cc|cc|}
  \hline
\multirow{2}{*}{$\alpha$}
& \multirow{2}{*}{$N$}
& \multicolumn{2}{|c|}{$\delta = 10^{-6} $}
  & \multicolumn{2}{|c|}{$\delta = \pi/6$} & \multicolumn{2}{|c|}{$\delta = 2.5 h$} 
  & \multicolumn{2}{|c|}{$\delta = \sqrt{h}$}
 \\
  \cline{3-10}
 &  & $L^2$ error & order & $L^2$ error & order & $L^2$ error & order
 & $L^2$ error & order  \\ \hline\hline
& 24   & 4.361E-04 &  --       &  2.196E-04  &  --         &   1.971E-04  &  --        &  2.000E-04  &  --         \\
& 36   & 1.963E-04 & 1.969 &  1.049E-04 &  1.823  &   8.513E-05  &  2.071 &  8.854E-05  & 2.009   \\
& 48   & 1.109E-04 & 1.984 &  6.214E-05 &  1.819  &   4.732E-05  &  2.041 &  4.979E-05  & 2.001   \\
$ \frac 12$
& 60   & 7.114E-05 & 1.990 &  4.126E-05 &  1.835  &   3.010E-05  &  2.027 &  3.194E-05  & 1.989   \\
& 72   & 4.946E-05 & 1.994 &  2.942E-05 &  1.855  &   2.083E-05  &  2.019 &  2.221E-05  & 1.993   \\
& 84   & 3.636E-05 & 1.995 &  2.205E-05 &  1.872  &   1.527E-05  &  2.013 &  1.636E-05  & 1.985   \\ 
& 96   & 2.785E-05 & 1.997 &  1.714E-05 &  1.886  &   1.168E-05  &  2.010 &  1.254E-05  & 1.988   \\ 
   \hline
& 24   & 4.361E-04 &  --       &  2.207E-04  &  --         &   2.347E-04  &  --        &  2.311E-04  &  --         \\
& 36   & 1.963E-04 & 1.969 &  9.316E-05 &  2.127  &   1.036E-04  &  2.016 &  9.897E-05  & 2.092   \\
& 48   & 1.109E-04 & 1.984 &  5.096E-05 &  2.097  &   5.815E-05  &  2.008 &  5.443E-05  & 2.078   \\
$ \frac 52$
& 60   & 7.114E-05 & 1.990 &  3.205E-05 &  2.078  &   3.717E-05  &  2.005 &  3.430E-05  & 2.070   \\
& 72   & 4.946E-05 & 1.994 &  2.199E-05 &  2.067  &   2.580E-05  &  2.004 &  2.354E-05  & 2.064   \\
& 84   & 3.636E-05 & 1.995 &  1.602E-05 &  2.056  &   1.895E-05  &  2.002 &  1.714E-05  & 2.059   \\ 
& 96   & 2.785E-05 & 1.997 &  1.218E-05 &  2.049  &   1.450E-05  &  2.002 &  1.303E-05  & 2.055   \\ 
   \hline
\end{tabular}
\end{table}

\begin{table}[!htbp] \small
\caption{\label{tab_exp3_k2} $L^2$ errors and convergence orders produced by the scheme 
\eqref{dgscm_td} and \eqref{eqn_nip} when $k = 2$ in Example \ref{exp3}, 
$\sigma = 1/2$.}
\centering 
\begin{tabular}{|c|c||cc|cc|cc|cc|}
  \hline
\multirow{2}{*}{$\alpha$}
& \multirow{2}{*}{$N$}
& \multicolumn{2}{|c|}{$\delta = 10^{-6} $}
  & \multicolumn{2}{|c|}{$\delta = \pi/6$} & \multicolumn{2}{|c|}{$\delta = 2.5 h$} 
  & \multicolumn{2}{|c|}{$\delta = \sqrt{h}$}
 \\
  \cline{3-10}
 &  & $L^2$ error & order & $L^2$ error & order & $L^2$ error & order
 & $L^2$ error & order  \\ \hline\hline
& 24   & 1.973E-05 &  --       &  1.386E-05  &  --         &   1.380E-05  &  --        &  1.393E-05  &  --         \\
& 36   & 5.900E-06 & 2.977 &  3.985E-06 &  3.074  &   3.971E-06  &  3.072 &  4.160E-06  & 2.981   \\
& 48   & 2.498E-06 & 2.987 &  1.635E-06 &  3.096  &   1.627E-06  &  3.101 &  1.731E-06  & 3.048   \\
$ \frac 12$
& 60   & 1.282E-06 & 2.991 &  8.219E-07 &  3.083  &   8.166E-07  &  3.090 &  8.818E-07  & 3.022   \\
& 72   & 7.427E-07 & 2.993 &  4.698E-07 &  3.068  &   4.663E-07  &  3.073 &  5.021E-07  & 3.089   \\
& 84   & 4.681E-07 & 2.995 &  2.933E-07 &  3.055  &   2.910E-07  &  3.060 &  3.157E-07  & 3.011   \\ 
& 96   & 3.138E-07 & 2.996 &  1.953E-07 &  3.046  &   1.936E-07  &  3.049 &  2.091E-07  & 3.084   \\ 
   \hline
& 24   & 1.090E-05 &  --       &  9.065E-06  &  --         &   9.029E-06  &  --        &  9.035E-06  &  --         \\
& 36   & 3.184E-06 & 3.034 &  2.675E-06 &  3.010  &   2.659E-06  &  3.016 &  2.663E-06  & 3.013   \\
& 48   & 1.336E-06 & 3.019 &  1.125E-06 &  3.009  &   1.118E-06  &  3.012 &  1.120E-06  & 3.011   \\
$ \frac 52$
& 60   & 6.821E-06 & 3.013 &  5.752E-07 &  3.008  &   5.711E-07  &  3.009 &  5.723E-07  & 3.009   \\
& 72   & 3.941E-07 & 3.009 &  3.324E-07 &  3.008  &   3.300E-07  &  3.008 &  3.307E-07  & 3.008   \\
& 84   & 2.479E-07 & 3.007 &  2.091E-07 &  3.007  &   2.076E-07  &  3.007 &  2.080E-07  & 3.007   \\ 
& 96   & 1.659E-07 & 3.006 &  1.400E-07 &  3.007  &   1.390E-07  &  3.006 &  1.392E-07  & 3.006   \\ 
   \hline
\end{tabular}
\end{table}

\begin{table}[!htbp] \small
\caption{\label{tab_exp3_k3} $L^2$ errors and convergence orders produced by the scheme 
\eqref{dgscm_td} and  \eqref{eqn_nip} when $k = 3$ in Example \ref{exp3}, 
$\sigma = 1/2$.}  \centering 
\begin{tabular}{|c|c||cc|cc|cc|cc|}
  \hline
\multirow{2}{*}{$\alpha$}
& \multirow{2}{*}{$N$}
& \multicolumn{2}{|c|}{$\delta = 10^{-6} $}
  & \multicolumn{2}{|c|}{$\delta = \pi/6$} & \multicolumn{2}{|c|}{$\delta = 2.5 h$} 
  & \multicolumn{2}{|c|}{$\delta = \sqrt{h}$}
 \\
  \cline{3-10}
 &  & $L^2$ error & order & $L^2$ error & order & $L^2$ error & order
 & $L^2$ error & order  \\ \hline\hline
& 24   & 5.539E-07 &  --       &  3.290E-07  &  --         &   3.051E-07  &  --        &  3.083E-07  &  --       \\
& 36   & 1.060E-07 & 4.078 &  6.818E-08 &  3.882  &   5.843E-08  &  4.076 &  5.986E-08  & 4.042   \\
& 48   & 3.256E-08 & 4.103 &  2.258E-08 &  3.842  &   1.829E-08  &  4.037 &  1.887E-08  & 4.014   \\
$ \frac 12$
& 60   & 1.307E-08 & 4.091 &  9.575E-09 &  3.844  &   7.455E-09  &  4.022 &  7.724E-09  & 4.002   \\
& 72   & 6.216E-09 & 4.075 &  4.740E-09 &  3.857  &   3.586E-09  &  4.015 &  3.726E-09  & 3.999   \\
& 84   & 3.324E-09 & 4.061 &  2.610E-09 &  3.872  &   1.932E-09  &  4.010 &  2.014E-09  & 3.992   \\ 
& 96   & 1.936E-09 & 4.050 &  1.553E-09 &  3.886  &   1.132E-09  &  4.008 &  1.181E-09  & 3.994   \\ 
   \hline
& 24   & 5.539E-07 &  --       &  3.557E-07  &  --         &   3.558E-07  &  --        &  3.557E-07  &  --       \\
& 36   & 1.060E-07 & 4.078 &  7.051E-08 &  3.991  &   7.055E-08  &  3.990 &  7.052E-08  & 3.991   \\
& 48   & 3.256E-08 & 4.103 &  2.233E-08 &  3.996  &   2.235E-08  &  3.995 &  2.234E-08  & 3.996   \\
$ \frac 52$
& 60   & 1.307E-08 & 4.091 &  9.152E-09 &  3.998  &   9.162E-09  &  3.997 &  9.154E-09  & 3.998   \\
& 72   & 6.216E-09 & 4.075 &  4.415E-09 &  3.999  &   4.420E-09  &  3.998 &  4.416E-09  & 3.999   \\
& 84   & 3.324E-09 & 4.061 &  2.383E-09 &  3.999  &   2.387E-09  &  3.997 &  2.384E-09  & 3.998   \\ 
& 96   & 1.936E-09 & 4.050 &  1.397E-09 &  3.999  &   1.400E-09  &  3.995 &  1.398E-09  & 3.998   \\ 
   \hline
\end{tabular}
\end{table}

\begin{exmp} \label{exp4}
We consider the convection-diffusion problem in Example \ref{exp3} with piecewise constant initial values 
\begin{align*}
u_0(x) = \left\{
\begin{aligned}
u_l, & \qquad x < 0, \\
u_r, & \qquad  x > 0 .
\end{aligned}
\right.
\end{align*}
We consider two kinds of initial conditions:
\begin{enumerate}[{\rm (i)}]
\item $u_l = 0$, $u_r = 1$.
\item $u_l = 1$, $u_r = 0$.
\end{enumerate}
To see the nonlocal diffusion effect, we take four different values of $\sigma$. 
We take locally integrable kernel with $\alpha = 1/2$, $\delta = 1/8$. 
The computational domain is $\Omega = (-9, 9)$, and the final time is $T = 2$. 
\end{exmp}

From Table \ref{tab_exp4}, we can see the errors and orders of convergence for the Riemann problems in Example \ref{exp4}. The reference solution is computed with the refined mesh $h = 1/50$ and the $L^2$ errors are computed by the absolute errors $\| u_h - u_{\rm ref}\|_{L^2([-2, 6])}$. 
 The orders are around 3 and even higher when $h$ becomes smaller. 
 This phenomenon may be because  
 the spatial error is small compared to the temporal error as the mesh is refined. 
 Figure \ref{fig_exp4} plots  the numerical solution produced by the 
scheme \eqref{dgscm_td} and \eqref{eqn_nip}, from which we can see the numerical solutions become smooth with the nonlocal diffusion term, compared to the one without the nonlocal diffusion. 
We can also see that the transition width becomes larger as we increase the coefficient of the nonlocal diffusion term. This indicates that nonlocal diffusion can be another option when adding artificial diffusion to numerical schemes, especially for physical problems containing long-range interactions.

\begin{table}[!htbp] \small
 \caption{\label{tab_exp4} $L^2$ errors and convergence orders produced by the scheme 
\eqref{dgscm_td} and \eqref{eqn_nip} when $k = 2$ in Example \ref{exp4}, $\delta = 1/8$, 
$\sigma = 1/5$. The $L^2$ errors are computed on $[-2, 6]$. }
\centering 
\begin{tabular}{|c|c||c|c|c|c|c|c|c|}
  \hline
\multicolumn{2}{|c||}{$h$} & 1/2 & 1/3 & 1/4 & 1/5 & 1/6 & 1/7 & 1/8   \\ \hline
\multirow{2}{*}{I.C. (i)} 
& $L^2$ error  &2.223E-04   & 6.897E-05 &  2.913E-05  &    1.456E-05    &  8.084E-06   &  4.815E-06    &  2.997E-06    \\
&  order            & --   &  2.886       &  2.996          &  3.107             & 3.228             &   3.362           & 
  3.549   \\  \hline
\multirow{2}{*}{I.C. (ii)} 
& $L^2$ error  &2.223E-04   & 6.897E-05 &  2.913E-05  &    1.456E-05    &  8.084E-06   &  4.815E-06    &  2.997E-06    \\
&  order            & --   &  2.886       &  2.996          &  3.107             & 3.228             &   3.362           & 
  3.549   \\  \hline
\end{tabular} 
\end{table}

\begin{figure}[!htbp]
\subfigure[I.C. (i)]{ 
\includegraphics[width=0.45\textwidth]{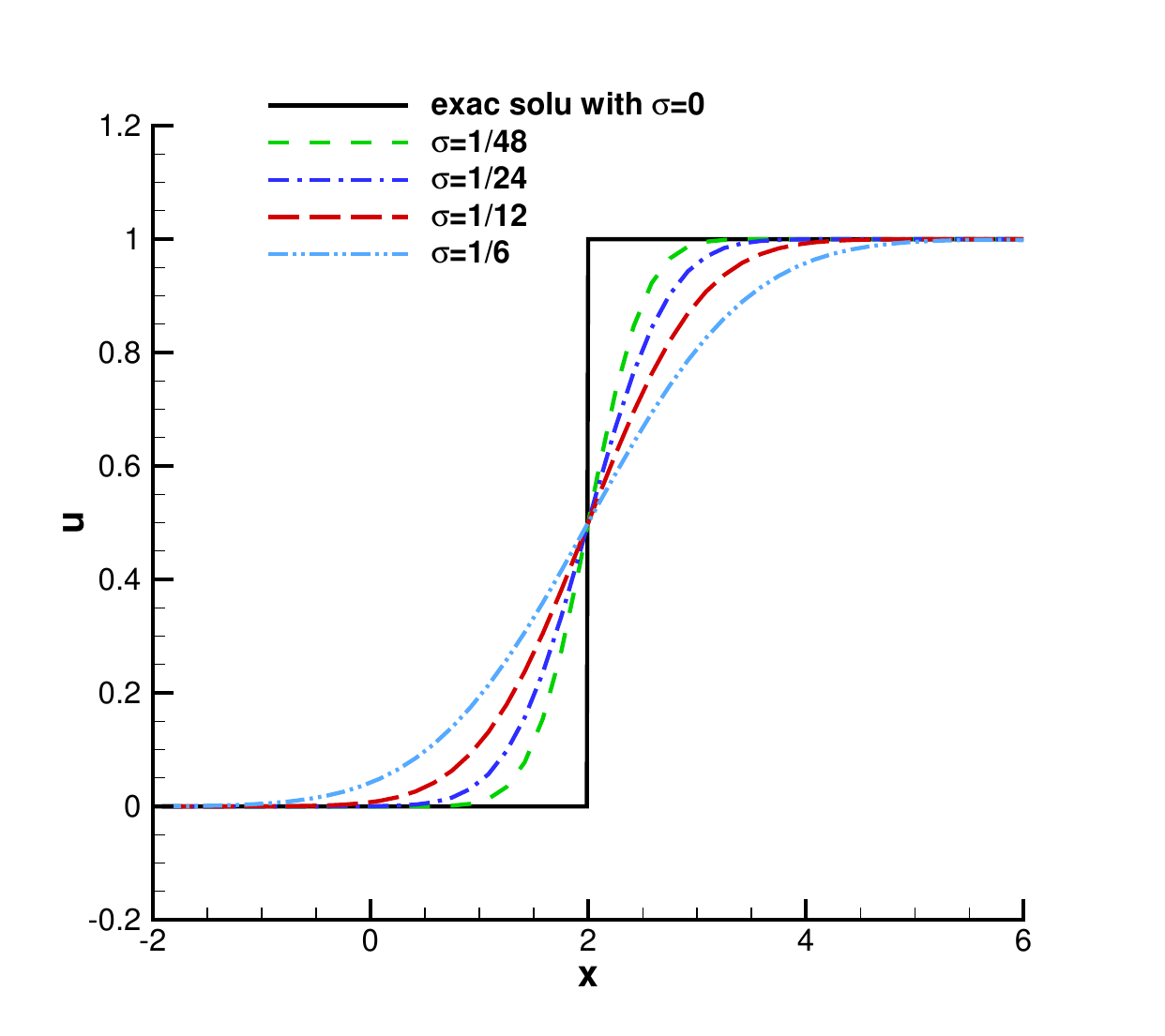} }
\subfigure[I.C. (ii)]{
\includegraphics[width=0.45\textwidth]{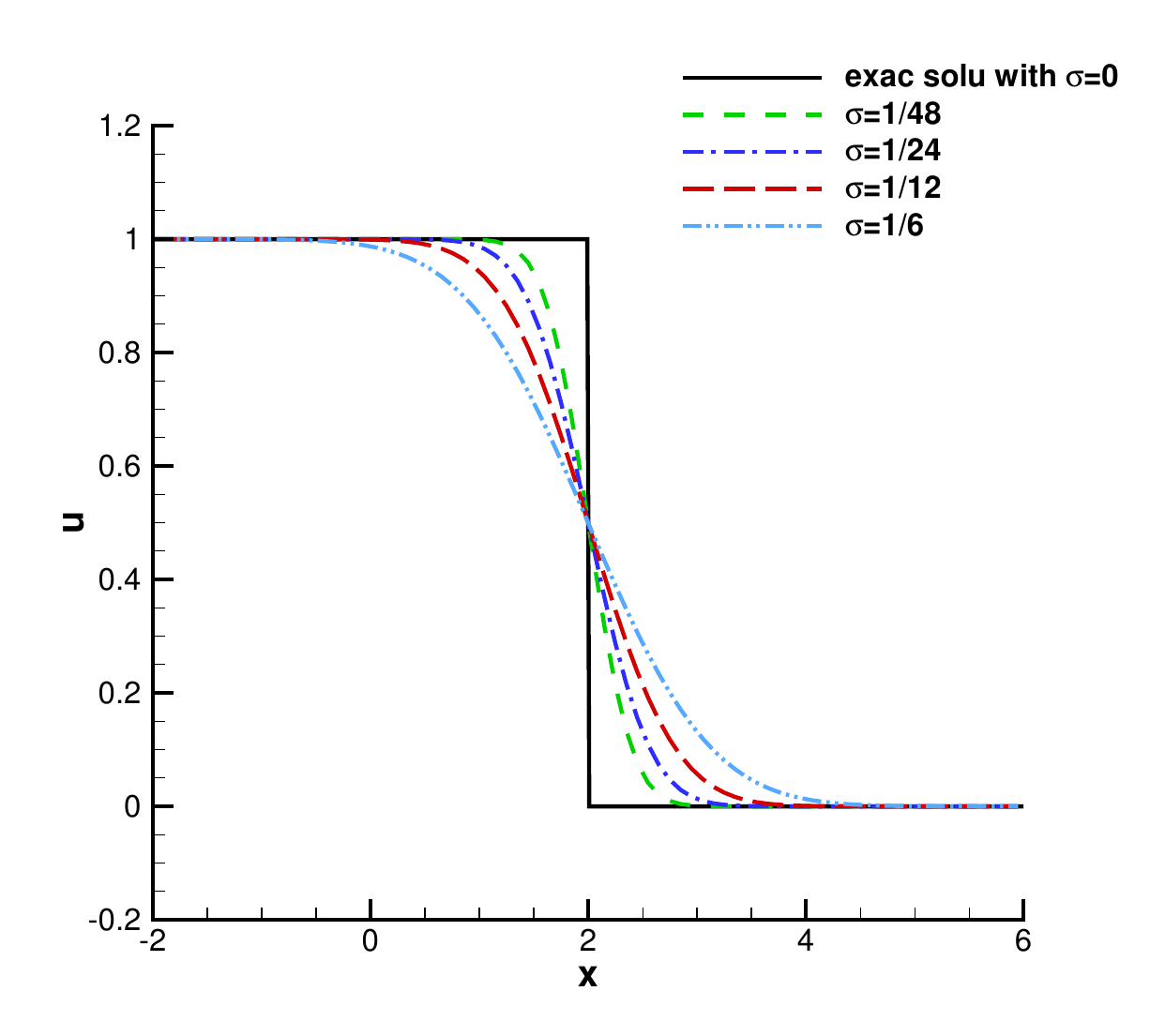} }
\caption{ \label{fig_exp4} 
Plots of the numerical solution in Example \ref{exp4} produced by the 
scheme \eqref{dgscm_td} and \eqref{eqn_nip}, $\delta = 1/8, \alpha = 1/2$, $k = 2$, $h = 1/6$. 
Solid line: exact solution with $\sigma = 0$. Green dash line: $\sigma = 1/48$. 
Blue dash-dot line: $\sigma = 1/24$. Red long dash line: $\sigma = 1/12$. 
Light blue dash dot dot line: $\sigma = 1/6$. } 
\end{figure}

\begin{exmp} \label{exp5}
Consider the viscous Burgers' equation with $f(u) = u^2/2$ in \eqref{eqn_cd}. 
We take the locally integrable kernel with $\alpha = 1/2$, $\delta = \pi/6$. The computational domain is 
$\Omega = (0, 2 \pi)$ and the final time is $T = 1.6$. Note that at time $T = 1.6$, the exact solution of the inviscid 
Burgers' equation {\rm($\sigma = 0$)} contains a shock discontinuity inside the domain. 
\end{exmp}

In Example \ref{exp5}, we investigate the nonlocal diffusion term for a convection-diffusion equation with the nonlinear convective term.  From Table \ref{tab_exp5}, 
we can see the errors and orders of convergence for the Example \ref{exp5}. 
The reference solution is computed with the refined mesh $h = \pi/250$ and the $L^2$ errors are computed by the absolute errors $\| u_h - u_{\rm ref}\|_{L^2([0, 2 \pi])}$. 
The convergence rate is around 2 and is not optimal, which could be associated with 
the lack of regularity of the exact solution. 
 We still observe the same phenomenon as in Example \ref{exp4} from the numerical solutions 
 plotted in Figure \ref{fig_exp5}, which again demonstrates the effect of the nonlocal diffusion.

\begin{table}[!htbp] \small
 \caption{\label{tab_exp5} $L^2$ errors and convergence orders produced by the scheme 
\eqref{dgscm_td} and \eqref{eqn_nip} when $k = 2$ in Example \ref{exp5}, $\delta = \pi/6$, 
$\sigma = \pi/15$.}
\centering 
\begin{tabular}{|c||c|c|c|c|c|c|c|}
  \hline
$N$ & 48 & 60 & 72 & 84 & 96 & 108 & 120 \\ \hline
$L^2$ error  &3.965E-05   & 2.903E-05 &  2.238E-05  &    1.731E-05    &  1.334E-05   &  1.040E-05    &  8.351E-05    \\
 order            & --   &  1.397       &  1.426          &  1.667             & 1.951             &   2.113           & 
  2.083   \\  \hline
\end{tabular} 
\end{table}

\begin{figure}[!htbp]
\includegraphics[width=0.6\textwidth]{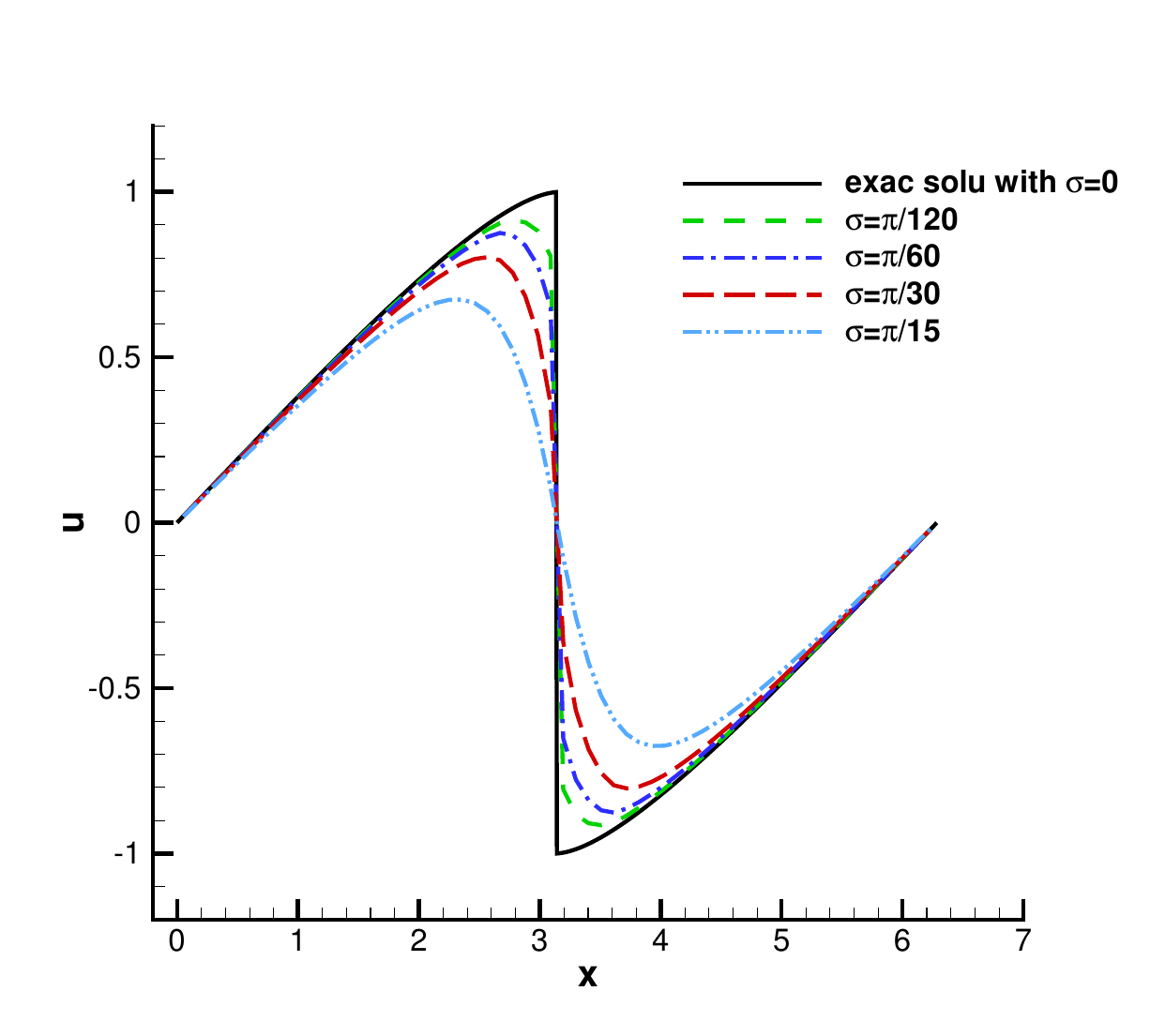} 
\caption{ \label{fig_exp5} 
Plot of the numerical solution in Example \ref{exp5} produced by the 
scheme \eqref{dgscm_td} and \eqref{eqn_nip}, $\delta = \pi/6, \alpha = 1/2$, $k = 2$, $N=60$. 
Solid line: exact solution with $\sigma = 0$. Green dash line: $\sigma = \pi/120$. 
Blue dash dot line: $\sigma = \pi/60$. Red long dash line: $\sigma = \pi/30$. 
Light blue dash dot dot line: $\sigma = \pi/15$. } 
\end{figure}

\section{Conclusion}
\label{sec_summary}

In this paper, we propose a class of penalty discontinuous Galerkin (DG) methods for nonlocal 
diffusion (ND) problems. The ND problem is an integral equation with a possibly singular kernel, which 
can describe physical systems with long-range interactions and allow solutions with low regularity. 
Therefore, it is natural to adopt the DG method to compute the ND problem when the solution contains 
singularities and discontinuities. The proposed DG methods in this paper have corresponding local counterparts, 
indicating that these methods work at least for the vanished horizon. We also present theoretical results 
on boundedness, stability, a priori error estimates, and asymptotic compatibility. To illustrate the nonlocal 
diffusion effect, we consider the time-dependent convection-diffusion equation with nonlocal diffusion 
and conduct a semi-discrete analysis. Numerical tests demonstrate that our methods are high-order, 
stable, and asymptotically compatible. Currently, the DG methods are developed for one-dimensional 
ND problems, and extending these methods to multidimensional problems is part of our future work.

\begin{appendix}

\section{The technical proof of Lemma \ref{lem_sta_cond}} 
\label{sec_proof_lem_sta_cond}

\begin{proof}
First let us consider two cases: $0 < s < \varepsilon \rho $ and $\varepsilon \rho < 
s < \hat{h}$, where $ \varepsilon \in (0, 1) $ will be determined later. 
In fact, with the definition of $|\cdot|_{J, h}$ in \eqref{eqn_seminorm}, we have
\begin{align} \label{eqn_sta_J_1}
|v_h|_{J, h}^2 = J_1 + J_2, 
\end{align}
where $J_1$ and $J_2$ are defined as 
\begin{align*}
\begin{aligned}
 J_1 &  \, = 2 \sum\limits_{j} h_j \int_0^{\varepsilon \rho} \gamma_\delta(s) 
\, \dfrac{1}{s} \int_{I_{j, 2}^s} \big(E_s^+ v_h 
- [\![ v_h ]\!]_{j + \frac 12} \big)^2 \, \dd x \dd s,  \\
J_2 & \, = 2 \sum\limits_{j} h_j \int_{\varepsilon \rho}^{\hat{h}} \gamma_\delta(s) 
\, \dfrac{1}{s} \int_{I_{j, 2}^s} \big(E_s^+ v_h 
- [\![ v_h ]\!]_{j + \frac 12} \big)^2 \, \dd x \dd s .  
\end{aligned}
\end{align*}
We now consider $J_1$, i.e. the case $0 < s < \varepsilon \rho $, and we have 
\begin{align} \label{eqn_sta_J_2}
\begin{aligned}
J_1 & \, \leq 4 \sum\limits_{j} h_j \int_0^{\varepsilon \rho} \gamma_\delta(s) 
\, \dfrac{1}{s} \int_{I_{j, 2}^s} \big(v_h(x + s) 
- v_h(x_{j + \frac 12}^+ ) \big)^2 \, \dd x \dd s \\
& \quad + 4 \sum\limits_{j} h_j \int_0^{\varepsilon \rho} \gamma_\delta(s) 
\, \dfrac{1}{s} \int_{I_{j, 2}^s} \big(v_h(x ) 
- v_h(x_{j + \frac 12}^- ) \big)^2 \, \dd x \dd s. 
\end{aligned}
\end{align}
In fact, we have 
\begin{align} \label{eqn_sta_J_3}
\begin{aligned}
 \int_{I_{j, 2}^s} \big(v_h(x + s) - v_h(x_{j + \frac 12}^+ ) \big)^2 \,\dd x 
& \, \leq  \int_{I_{j, 2}^s} \Big(\int_{x_{j + \frac 12}}^{x + s} \partial_y v_h(y)\,\dd y \Big)^2 
\,\dd x \\
& \, \leq \int_{I_{j, 2}^s} \Big(\int_{x_{j + \frac 12}}^{x + s}  1 \,\dd y \Big)^2 \,\dd x \,  
 \| \partial_y v_h \|_{L^{\infty}(I_{j+1, 2}^s )}^2 \\
 & \, = \frac 13 s^3 \,  \| \partial_y v_h \|_{L^{\infty}(I_{j+1, 2}^s )}^2 . 
\end{aligned}
\end{align}
Similarly, we also have 
\begin{align} \label{eqn_sta_J_4}
\int_{I_{j, 2}^s} \big(v_h(x ) - v_h(x_{j + \frac 12}^- ) \big)^2 \,\dd x 
\leq \frac 13 s^3 \,  \| \partial_y v_h \|_{L^{\infty}(I_{j , 2}^s )}^2 . 
\end{align}
Plug \eqref{eqn_sta_J_3} and \eqref{eqn_sta_J_4} into \eqref{eqn_sta_J_2}, with 
inverse estimates we then obtain 
\begin{align} \label{eqn_sta_J_5}
\begin{aligned}
J_1  & \, \leq  \frac 43 \int_0^{\varepsilon \rho} 
s^2 \gamma_\delta(s) \,\dd s\,  \sum\limits_{j} h_j  
\big( \| \partial_y v_h \|_{L^{\infty}(I_{j+1, 2}^s )}^2 
+ \| \partial_y v_h \|_{L^{\infty}(I_{j, 2}^s )}^2 \big)  \\
& \, \leq C  \int_0^{\varepsilon \rho} s^2 \gamma_\delta(s) \,\dd s 
 \sum\limits_{j}  \| \partial_y v_h \|_{L^{2}(I_j )}^2 .   
 \end{aligned}
\end{align}


Next we consider the $|\cdot|_{E, h}$ defined in \eqref{eqn_seminorm}. 
Similar to \eqref{eqn_sta_J_1}, we divide the $| v_h |_{E, h}^2 $ into three parts: 
\begin{align} \label{eqn_sta_E_1}
|v_h|_{E, h}^2  = E_1 + E_2 + E_3,  
\end{align}
where $E_1 = E_{1,1} + E_{1,2}$, $E_2 = E_{2,1} + E_{2,2}$, and $E_3$ are given as 
\begin{align*}
\begin{aligned}
 E_{1,1} & \, = 2 \sum\limits_{j} \int_0^{\varepsilon \rho} \gamma_\delta(s) 
\int_{I_{j,1}^s} g_{v_h}(x, s)^2 \, \dd x \dd s , \quad  
E_{1,2} =  \sum\limits_{j} \int_0^{\varepsilon \rho} 
\gamma_\delta(s) \int_{I_{j,2}^s} g_{v_h}(x, s)^2 \, \dd x \dd s , \\
 E_{2, 1} & \, = 2 \sum\limits_{j} \int_{\varepsilon \rho}^{\hat{h}} \gamma_\delta(s) 
\int_{I_{j,1}^s} g_{v_h}(x, s)^2 \, \dd x \dd s, \quad 
E_{2, 2}  = 2 \sum\limits_{j} \int_{\varepsilon \rho}^{\hat{h}} \gamma_\delta(s) 
\int_{I_{j,2}^s} g_{v_h}(x, s)^2 \, \dd x \dd s,  \\
E_3 & \, = 2 \sum\limits_{j} \int_{\hat{h}}^{\delta} \gamma_\delta(s) 
\int_{I_j} g_{v_h}(x, s)^2 \, \dd x \dd s. 
\end{aligned}
\end{align*}
We claim that for $\varepsilon$ sufficiently small but independent of 
$h$ and $\delta$, we have 
\begin{align} \label{eqn_sta_E_2}
E_{1,1} \geq C \, \int_0^{\varepsilon \rho} s^2 \gamma_\delta(s) \,\dd s 
 \sum\limits_{j}  \| \partial_y v_h \|_{L^{2}(I_j )}^2 .  
\end{align}
 If \eqref{eqn_sta_E_2} holds, 
then $J_1$ can be bounded by $E_{1, 1}$. Now we prove that \eqref{eqn_sta_E_2} 
is true when $\varepsilon$ is appropriately chosen.  

From the definition of $g_{v_h}(x, s)$ in \eqref{eqn_gv}, we know that 
$g_{v_h}(x, s) = E_s^+ v_h(x)$ on $I_{j, 1}^s \times (0, \varepsilon \rho)$. 
By Taylor expansion, we obtain 
$$
E_s^+ v_h(x) = v_h(x + s) - v_h(x) = \sum\limits_{l = 1}^k \frac{\partial_x^{l} v_h (x)}{l!} s^l.  
$$
Thus, with Cauchy-Schwarz inequality, we have 
\begin{align} \label{eqn_sta_E_3}
\begin{aligned}
\int_{I_{j,1}^s} g_{v_h}(x, s)^2 \,\dd x & \, = \int_{I_{j,1}^s} \Big( \sum\limits_{l = 1}^k 
\frac{\partial_x^{l} v_h (x)}{l!} s^l \Big)^2 \,\dd x \\
& \, = \int_{I_{j,1}^s} \big( \partial_x v_h(x) s \big)^2 \,\dd x 
+ 2 \int_{I_{j,1}^s} \big( \partial_x v_h(x) s \big) \Big( \sum\limits_{l = 2}^k 
\frac{\partial_x^{l} v_h (x)}{l!} s^l \Big)  \,\dd x \\
& \quad + \int_{I_{j,1}^s}  \Big( \sum\limits_{l = 2}^k 
\frac{\partial_x^{l} v_h (x)}{l!} s^l \Big)^2  \,\dd x \\
& \, \geq \frac 12 \, s^2  \int_{I_{j,1}^s} \big( \partial_x v_h(x) \big)^2 \,\dd x 
-  \int_{I_{j,1}^s}  \Big( \sum\limits_{l = 2}^k 
\frac{\partial_x^{l} v_h (x)}{l!} s^l \Big)^2  \,\dd x . 
\end{aligned}
\end{align}
Note that $s \in (0, \varepsilon \rho)$, then with Cauchy-Schwarz inequality and inverse estimates we have 
\begin{align} \label{eqn_sta_E_4}
\begin{aligned}
& \int_{I_{j,1}^s}  \Big( \sum\limits_{l = 2}^k \frac{\partial_x^{l} v_h (x)}{l!} s^l \Big)^2  \,\dd x \\
& \quad = \int_{I_{j,1}^s}  \Big( \sum\limits_{l = 2}^k \partial_x^{l} v_h (x) (h_j - s)^{l - 1} \cdot 
 (h_j - s)^{- l + 1} \frac{s^l}{l!}  \Big)^2  \,\dd x  \\
& \quad \leq   s^2 \bigg( \sum\limits_{l = 2}^k \frac{1}{(l!)^{2}}  \Big( \frac{s}{h_j - s} 
\Big)^{2l - 2} \bigg) 
\bigg( \sum\limits_{l = 2}^k  (h_j - s)^{2l - 2} \int_{I_{j,1}^s} 
 \big( \partial_x^{l} v_h (x) \big)^2\dd x \bigg) \\
 & \quad \leq C_2 \bigg( \sum\limits_{l = 2}^k  \Big( \frac{\varepsilon}{1 
 - \varepsilon} \Big)^{2l - 2} \bigg) s^2 \int_{I_{j,1}^s}  \big( \partial_x v_h (x) \big)^2\dd x, 
\end{aligned}
\end{align}
where $C_2 > 0$ is independent of $h$ and $\delta$. 
Since 
$$
\sum\limits_{l = 2}^k  \Big( \frac{\varepsilon}{1 
 - \varepsilon} \Big)^{2l - 2} = \varepsilon_1 \frac{1 - \varepsilon_1^{k - 1}}{1 - \varepsilon_1}, 
 \quad \varepsilon_1 = \Big( \frac{\varepsilon}{1 - \varepsilon} \Big)^2, 
$$
we then take $\varepsilon $ statisfying $ \varepsilon \leq 1/(1 + \sqrt{4 C_2 + 1}) $ such that 
$\varepsilon_1 \leq 1/(4 C_2 + 1)$, which implies 
$$
  C_2 \varepsilon_1 \frac{1 - \varepsilon_1^{k - 1}}{1 - \varepsilon_1} \leq 
  \frac{C_2 \varepsilon_1}{1 - \varepsilon_1}  \leq \frac 14 \;. 
$$
Therefore, plugging \eqref{eqn_sta_E_4} into \eqref{eqn_sta_E_3} with small $\varepsilon$ 
satisfying the above inequality, we obtain 
\begin{align} \label{eqn_sta_E_5}
\int_{I_{j,1}^s} g_{v_h}(x, s)^2 \,\dd x \geq  \frac 14 \, s^2  \int_{I_{j,1}^s} 
\big( \partial_x v_h(x) \big)^2 \,\dd x .  
\end{align}
If we take the Taylor expansion of $E_s^+ v_h(x)$ that 
$$
E_s^+ v_h(x) = - \sum\limits_{l = 1}^k \frac{(\partial_x^{l} v_h) (x + s)}{l!} (-s)^l,   
$$
we can still obtain a similar result in the following. 
\begin{align} \label{eqn_sta_E_6}
\int_{I_{j,1}^s} g_{v_h}(x, s)^2 \,\dd x \geq  \frac 14 \, s^2  \int_{I_{j,1}^s} 
\big( (\partial_x v_h)(x + s) \big)^2 \,\dd x .  
\end{align}
Therefore, with \eqref{eqn_sta_E_5} and \eqref{eqn_sta_E_6} we can obtain
\begin{align} \label{eqn_sta_E_7}
\begin{aligned}
E_{1,1} & \, = \sum\limits_{j} \int_0^{\varepsilon \rho} \gamma_\delta(s) 
\Big( 2 \int_{I_{j,1}^s} g_{v_h}(x, s)^2 \,\dd x \Big) \dd s \\
&\, \geq \frac 14  \,\int_0^{\varepsilon \rho} s^2 \gamma_\delta(s) 
 \sum\limits_{j} \Big(\int_{I_{j,1}^s} 
\big( \partial_x v_h(x) \big)^2 \,\dd x + \int_{I_{j,1}^s} 
\big( (\partial_x v_h)(x + s) \big)^2 \,\dd x  \Big) \dd s \\
& \, \geq \frac 14  \,\int_0^{\varepsilon \rho} s^2 \gamma_\delta(s) 
 \sum\limits_{j} \int_{I_{j}} \big( \partial_x v_h(x) \big)^2 \, \dd x \dd s ,  
 \end{aligned}
\end{align}
which is exactly the result we want, as we claim in \eqref{eqn_sta_E_2}.  
Now let us fix $\varepsilon = \varepsilon_0$ so that \eqref{eqn_sta_E_2} holds, 
we then consider the another case: $\varepsilon_0 \rho < s < \hat{h} $. 
 In fact, we have
 \begin{align} \label{eqn_sta_E_8}
 \begin{aligned}
 J_2 & \, \leq 2 \sum\limits_{j} h_j \int_{\varepsilon_0 \rho}^{\hat{h}} \gamma_\delta(s) 
\, \dfrac{1}{\varepsilon_0 \rho} \int_{I_{j, 2}^s} \big(E_s^+ v_h 
- [\![ v_h ]\!]_{j + \frac 12} \big)^2 \, \dd x \dd s  \\
& \, \leq \frac{2 \nu}{\varepsilon_0} \sum\limits_{j} h_j 
\int_{\varepsilon_0 \rho}^{\hat{h}} \gamma_\delta(s) 
\int_{I_{j, 2}^s} \big(E_s^+ v_h - [\![ v_h ]\!]_{j + \frac 12} \big)^2 \, \dd x \dd s 
= \frac{ \nu}{\varepsilon_0} E_{2,2},  
\end{aligned}
 \end{align}
 where $\nu$ is defined in \eqref{mesh_regularity}. 
Therefore, by \eqref{eqn_sta_J_5}, \eqref{eqn_sta_E_1}, \eqref{eqn_sta_E_2} and \eqref{eqn_sta_E_8}, 
we obtain \eqref{eqn_sta_cond} and complete the proof. 

\end{proof}

\end{appendix}

\end{document}